\documentclass[11pt]{amsart}
\usepackage{amssymb}
\usepackage{amsmath}
\usepackage{amsthm}
\usepackage{bm}
\usepackage{graphicx}
\usepackage{ascmac}
\usepackage{comment}
\usepackage{mathtools}
\usepackage{mathrsfs}
\usepackage{tikz}
\usetikzlibrary{positioning}
\usetikzlibrary{cd}
\usepackage{pgfplots}
\pgfplotsset{compat=newest}
\usepackage[all]{xy}
\usepackage{latexsym}
\usepackage{amsfonts}
\usepackage{hyperref}
\usepackage{cleveref}

\theoremstyle{definition}
\newtheorem{defn}{Definition}[section]

\newtheorem{example}[defn]{Example}

\theoremstyle{plain}
\newtheorem{thm}[defn]{Theorem}
\newtheorem{prop}[defn]{Proposition}
\newtheorem{lem}[defn]{Lemma}
\newtheorem{cor}[defn]{Corollary}

\theoremstyle{remark}
\newtheorem{rem}[defn]{Remark}

\newcommand{\Ho}{\mathrm{Ho}}
\newcommand{\perf}{\mathrm{Perf}}
\newcommand{\qcoh}{\mathrm{QCoh}}
\newcommand{\QCoh}{\mathrm{QCSh}}
\newcommand{\coh}{{\mathrm{Coh}}}
\newcommand{\Idem}{\mathrm{Idem}}
\newcommand{\Ind}{\mathrm{Ind}}
\newcommand{\Spec}{\mathrm{Spec}}
\newcommand{\Sub}{\mathrm{Sub}}

\newcommand{\tSpec}{\mathrm{Spec}_{\triangle}}
\newcommand{\mo}{\mathrm}
\newcommand{\ol}{\overline}
\newcommand{\os}{\overset}
\newcommand{\Supp}{\mathrm{Supp}}
\newcommand{\Nat}{\mathrm{Nat}}
\newcommand{\sO}{\mathscr{O}}
\newcommand{\Hom}{\mathrm{Hom}}
\newcommand{\ld}{\mathbb{L}}
\newcommand{\rd}{\mathbb{R}}
\newcommand{\dotimes}{\otimes^{\mathbb{L}}}

\newcommand{\sNat}{\mathcal{N}\hspace{-0.08cm}\mathit{at}}

\newcommand{\Perf}{\mathrm{Cat}_{\infty}^{\mathrm{perf}}}

\newcommand{\intg}{\mathbb{Z}}
\newcommand{\Id}{\mathrm{Id}}
\newcommand{\id}{\mathrm{id}}
\newcommand{\Sing}{\mathrm{Sing}}

\newcommand{\ang}[1]{{\left\langle{#1}\right\rangle}}
\newcommand{\End}{\mathrm{End}}

\newcommand{\Fib}{\mathrm{Fib}}

\newcommand{\oSpec}{\mathrm{Spec}_{\otimes}}

\newcommand{\etch}{\mathrm{H}_{\text{\'{e}t}}}
\newcommand{\Spc}{\mathrm{Spc}}
\newcommand{\fSpec}{{\mathrm{Spec}^{\mathrm{fin}}}}
\newcommand{\oSpc}{\mathrm{Spc}_{\otimes}}
\newcommand{\tSpc}{\mathrm{Spc}_{\triangle}}
\newcommand{\mapsfrom}{\mathrel{\reflectbox{\ensuremath{\longmapsto}}}}

\numberwithin{equation}{section}

\begin{document}

\title{The Spectrum of Stable Infinity Categories with Actions}

\author{Hisato Matsukawa}
\address{Department of Mathematics, Faculty of Science, Hokkaido University, Kita 10, Nishi 8, Kita-Ku, Sapporo, Hokkaido, 060-0810, Japan}
\email{matsukawa.hisato.f4@elms.hokudai.ac.jp, hmnr0211@gmail.com}
\thanks{The author is supported by JST SPRING, Grant Number JPMJSP2119.}

\subjclass[2010]{14F08, 14A20, 14A15, 14M10, 18G80, 18N60}
\keywords{Matsui spectrum, Balmer spectrum, stable infinity category, derived category}
\date{}

\dedicatory{}

\begin{abstract}
We introduce the relative Matsui spectrum, a new invariant associated with a stable \(\infty\)-category equipped with an action.  
This construction generalizes both Balmer's tensor triangular spectra and Matsui's triangular spectra, and provides a unified framework for classifying thick submodules.  
We establish its fundamental properties, including universality, comparison with existing spectra, and descent, and construct a natural morphism to the Balmer spectrum of the base.  
Applications show that the relative Matsui spectrum recovers the underlying classical geometric spaces from categorical data in various settings: categories of perfect complexes of schemes, twisted derived categories, categories of singularities, and derived matrix factorization categories.
Thus the relative Matsui spectrum extends the reach of tensor triangular geometry beyond globally tensorial settings, while preserving geometric intuition.
\end{abstract}

\maketitle
\tableofcontents

\section{Introduction}

The study of spectra associated with triangulated or stable \(\infty\)-categories has played a central role in bridging abstract categorical structures with underlying geometry.  
Balmer~\cite{Ba} introduced the tensor triangular spectrum \(\oSpec(T)\) of a tensor triangulated category \(T\), thereby providing a powerful framework for classifying thick tensor ideals.  
When \(T=\perf(X)\) for a topologically noetherian scheme \(X\), the spectrum \(\oSpec(T)\) recovers the scheme \(X\) together with its structure sheaf.  
More recently, Matsui proposed a variant of the spectrum that applies to triangulated categories without tensor structures, thereby expanding the range of situations where spectra can be defined.  
It has since been further studied by Hirano--Ouchi~\cite{HO, HO2}, Hirano--Kalck--Ouchi~\cite{HKO}, Ito~\cite{It}, and Ito--Matsui~\cite{IM}.

In this paper, we develop a relative version of the Matsui spectrum, which we call the \emph{relative Matsui spectrum}.  
Given a stable \(\infty\)-category \(D\) with an action of a stable \(\infty\)-category \(S\), we construct a locally ringed space
\[
\Spec(D,S),
\]
which simultaneously generalizes Balmer's and Matsui's construction.
This spectrum is the universal space classifying thick \(S\)-submodules of \(D\) and admits a canonical morphism to the Balmer spectrum of the base \(\oSpec(S)\).

We establish fundamental properties of the relative Matsui spectrum, including universality, comparison with existing spectra, and descent.
In particular, we show that, under suitable hypotheses, the relative Matsui spectrum recovers classical geometric objects.  
For example:
\begin{enumerate}
  \item For a rigid tensor triangulated \(\infty\)-category \(T\), taking \(D=S=T\) recovers Balmer spectrum.
  \item For a stable \(\infty\)-category \(D\) with \(S=\perf(\mathbb{Z})\) or \(\mo{Sp}^{\omega}\), it recovers Matsui spectrum. 
  \item For a flat morphism of schemes \(X\to Y\) with \(Y\) regular, the category \(\perf(X)\) admits a natural \(\perf(Y)\)-action, and the spectrum is a locally ringed space over \(Y\), fiberwise given by the Matsui spectrum of \(\perf(X_y)\).  
  Moreover, any relatively derived equivalent scheme appears as an open subscheme of the spectrum.
  \item For twisted derived categories \(\perf(X,\mu)\) with canonical \(\perf(X)\)-action, the spectrum is canonically isomorphic to \(X\).
  \item For a locally complete intersection scheme \(X\), the derived category of coherent sheaves \(\coh(X)\) admits a natural \(\perf(X)\)-action, and its spectrum is a space over \(X\).  
  This spectrum can be described explicitly.  
  For instance, the Krull dimension of each fiber coincides with the embedding codimension at the corresponding point.
  Submodules of \(\coh(X)\) correspond bijectively to specialization-closed subsets of the spectrum.
  \item For the category of singularities \(\Sing(X)\) with its \(\perf(X)\)-action, if \(X\) is locally a hypersurface, the spectrum recovers the singular locus of \(X\).
  \item For derived matrix factorization categories \(\mathrm{DMF}(X,L,W)\) with canonical \(\mathrm{DMF}(X,L,0)\)-action, the spectrum recovers the relative singular locus, extending results of Hirano.
  \item For a Severi--Brauer scheme \(X\to Y\), the spectrum encodes both \(X\) and additional copies of the base \(Y\), reflecting semi-orthogonal decompositions.  
  When the relative dimension is one, we describe the spectrum explicitly and classify its submodules.
\end{enumerate}

Thus the relative Matsui spectrum extends the reach of tensor triangular geometry to settings without a global tensor product, while preserving geometric intuition through the classification of subcategories.  

The paper is organized as follows.  
In Section~\ref{section:construction}, we construct the relative Matsui spectrum.  
In Section~\ref{section:universality}, we establish its universality.
In Section~\ref{section:fundamental_results}, we study its basic properties, including classification of submodules and comparison with existing spectra.  
In Section~\ref{section:commutativity_of_localizations}, we recall some fundamental properties of commutativity of localizations for later use.  
In Section~\ref{section:morphism_to_base}, we construct a morphism from the relative Matsui spectrum to the spectrum of the base and establish a fiberwise decomposition.  
In Section~\ref{section:applications}, we apply the theory to concrete examples: categories of perfect complexes of schemes, twisted derived categories, derived category of complexes with finitely generated homology over Koszul algebras, categories of singularities, derived categories of coherent sheaves, and derived matrix factorization categories.

\section*{Assumptions and Notations}

We adopt the same axioms concerning the existence of universes as in \cite{Lu1}.

All schemes appearing in this paper are assumed to be Noetherian.  
Let \(X\) be a scheme.  
We denote by \(\perf(X)\) the \(\infty\)-category of perfect complexes on \(X\).  
Similarly, \(\coh(X)\) denotes the bounded derived \(\infty\)-category of coherent sheaves on \(X\), and \(\qcoh(X)\) denotes the derived \(\infty\)-category of \(\sO_X\)-modules with quasi-coherent cohomology.  
If \(X\) is separated, then \(\qcoh(X)\) is equivalent to the derived \(\infty\)-category \(\mathrm{D}(\QCoh(X))\) of quasi-coherent sheaves on \(X\) (see \cite[Corollary~5.5]{BN}).  
We write \(\Sing(X)\) for the idempotent completion of the Verdier quotient \(\coh(X)/\perf(X)\), called the category of singularities of \(X\).  
The homotopy category of an \(\infty\)-category \(D\) is denoted by \(\Ho(D)\) or \(D_0\).

\(\Perf\) denote the \(\infty\)-category of idempotent complete stable \(\infty\)-categories.

\section{Construction}\label{section:construction}
In this section, we construct our spaces.
Starting from a triangulated category together with a family of thick subcategories, we define \(S\)-primes and quasi-\(S\)-primes and endow their sets with natural topologies.

Let \(T\) be a triangulated category, and let \(S\) be a set of thick subcategories of \(T\).
Assume that for any transfinite sequence of thick subcategories \( I_1 \subset I_2 \subset \cdots \in S\), the union \(\bigcup_i I_i\) belongs to \(S\).

\begin{defn}
  In the above setting, we define:
  \begin{enumerate}
    \item A thick subcategory \(P \in S\) is called a \emph{quasi-\(S\)-prime} if for any \(I_1, I_2 \in S\) with \(I_1, I_2 \supsetneq P\), we have \(I_1 \cap I_2 \supsetneq P\). 
    \item A thick subcategory \(P \in S\) is called an \emph{\(S\)-prime} if for any family \(\{I_i\}_{i \in I}\subset S\) with \(I_i \supsetneq P\), we have \(\bigcap_i I_i \supsetneq P\).
  \end{enumerate}
\end{defn}

For an \(S\)-prime \(P\), let \(\ol{P}\) denote the intersection \(\bigcap_{I\in S,\ I\supsetneq P} I\).
If \(S\) is closed under intersections, then \(\ol{P}\) is the smallest element of \(S\) strictly containing \(P\).

Let \(\Spc(T,S)\) denote the set of \(S\)-primes, and \(\Spc'(T,S)\) the set of quasi-\(S\)-primes.
We define a topology on \(\Spc(T,S)\) and \(\Spc'(T,S)\).
For a subset \(W \subset T\), the support of \(W\) in \(\Spc(T,S)\) is defined as
\[
\Supp(W):=\{P \in \Spc(T,S) \mid W \not\subset P\}.
\]
The family of closed subsets of \(\Spc(T,S)\) is generated by \(\{\Supp(a)\mid a\in T\}\).
Similarly for \(\Spc'(T,S)\).
There is a natural dense embedding \(\Spc(T,S)\hookrightarrow \Spc'(T,S)\).

\begin{rem}\label{rem:T_0_ness}
Let \(P\) and \(Q\) be \(S\)-primes.  
Then \(Q \in \overline{\{P\}}\) if and only if \(Q \subset P\).  
In particular, the space \(\Spec(D, S)\) is a \(T_0\)-space.
Moreover, for any \(W \subset T\), we have
\(
\Supp(W) = \Supp(\bigcap_{\substack{I\in S, W\subset I}} I)\), and this set is specialization-closed.
\end{rem}

There are naturally two presheaves of triangulated categories on \(\Spc(T,S)\) and \(\Spc'(T,S)\):
\begin{align*}
\mathcal{I}&: U \mapsto \{t\in T \mid \Supp(t)\cap U = \emptyset\} =: I_U,\\
\mathcal{T}&: U \mapsto \Idem(T/I_U) =: T_U,
\end{align*}
where \(\Idem\) denotes the idempotent completion and \(T/I_U\) the Verdier quotient.

\begin{lem}\label{lem:existance_of_prime}
Let \(I \in S\) and \(a \in T \setminus I\).
Then there exists an \(S\)-prime \(P \supset I\) such that \(a \notin P\).
\end{lem}
\begin{proof}
  By Zorn's lemma, the set \(\{J\in S \mid I\subset J,\ a\notin J\}\) has a maximal element \(P\) with respect to inclusion.
  Then \(a \in \bigcap_{I\in S,\ I\supsetneq P} I\), so \(P\) is an \(S\)-prime.
\end{proof}

\begin{lem}\label{lem:classification_of_submodules}
  The following map is injective:
  \[
  \Supp:S \longrightarrow 2^{\Spc(T,S)}.
  \]
  If \(S\) is closed under intersections, then its left inverse is given by 
  \[
\Spc(T,S)  \supset Z  \longmapsto \{a\in T \mid \Supp(a)\subset Z\}.
  \]
\end{lem}
\begin{proof}
This follows from Lemma \ref{lem:existance_of_prime}.
\end{proof}
\begin{example}
  \begin{enumerate}
    \item Let \(T\) be a tensor triangulated category and \(S\) the set of radical tensor ideals of \(T\).
    Then \(\Spc'(T,S)\) is the Balmer spectrum of \(T\) defined in \cite{Ba}.
    Indeed, a radical tensor ideal \(I\) is a prime ideal in the sense of Balmer if and only if it is quasi-\(S\)-prime.
    \item Let \(T\) be a triangulated category and \(S\) the set of thick subcategories of \(T\).
    Then \(\Spc(T,S)\) is the Matsui spectrum of \(T\) defined in \cite{Ma1}.
  \end{enumerate}
\end{example}

\begin{defn}
A \emph{datum} \((D,S,B,\phi)\) consists of algebra objects \(B,S \in \Perf\) with a homomorphism \(\phi:B \to S\) and a left \(S\)-module \(D \in \Perf\).
We denote the action of \(S\) on \(D\) by \(- \otimes -\).
When the context is clear, we simply write \((D,S,B)\) or \((D,S)\) instead of \((D,S,B,\phi)\).

A thick subcategory \(I \subset D\) is called an \(S\)-submodule if it is closed under the \(S\)-action; that is, \(s \otimes a \in I\) for all \(s \in S\), \(a \in I\).
The set of \(S\)-submodules of \(D\) is denoted by \(\Sub_S(D)\).
For a subset \(W \subset D\), we write \(\ang{W}_S\) for the minimal \(S\)-submodule containing \(W\).
\end{defn}

If \(W\) is closed under the \(S\)-action, then \(\ang{W}_S\) coincides with \(\ang{W}\), the minimal thick subcategory containing \(W\).

\begin{example}
  \begin{enumerate}
    \item Let \(f:X\to Y\) be a morphism of schemes.
    Then \(\perf(Y)\) acts on \(\perf(X)\) by
    \[
      G\otimes F := \ld f^* G \dotimes_{\sO_X} F
    \]
    for \(G\in \perf(Y), F\in \perf(X)\).
    Hence, \((\perf(X), \perf(Y), \perf(\mathbb{Z}))\) is a datum.
    Similarly, \(\coh(X)\) and \(\Sing(X)\) are \(\perf(Y)\)-modules and thus form data.

    \item Let \(X\) be a scheme, and let \(\mu=(\mu_0,\mu_1)\in \etch^2(X,\mathbb{G}_m)\times \etch^1(X,\mathbb{Z})\).
    Denote by \(\perf(X,\mu)\) the twisted derived category of perfect complexes defined in \cite{To2}.
    If \(\mu_1=0\), then \(\perf(X,\mu)\) is the derived category of perfect complexes of twisted sheaves (see \cite{Ca} for details).
    There is a natural action of \(\perf(X)\) on \(\perf(X,\mu)\), and \((\perf(X,\mu),\perf(X),\perf(\mathbb{Z}))\) is a datum.

    \item Let \(X\) be a scheme, \(L\) a line bundle on \(X\), and \(W\in \Gamma(X,L)\).
    Let \(\mo{DMF}(X,L,W)\) be the derived category of matrix factorizations defined in \cite{Pos, EP}.
    Then \(\mo{DMF}(X,L,W)\) is naturally a \(\mo{DMF}(X,L,0)\)-module, and \((\mo{DMF}(X,L,W),\mo{DMF}(X,L,0),\perf(\mathbb{Z}))\) is a datum.

    \item Let \(D\) be a stable \(\infty\)-category.
    Then \(D\) is naturally a module over the \(\infty\)-category of finite spectra \(\mo{Sp}^{\omega}\), and \((D,\mo{Sp}^{\omega},\mo{Sp}^{\omega})\) is a datum.

    \item Let \(k\) be a commutative ring, and let \(D_{\mo{dg}}\) be a \(k\)-linear dg-category.
    Then the associated stable \(\infty\)-category \(D\) has a natural \(\perf(k)\)-action, so \((D, \perf(k), \perf(k))\) is a datum.
  \end{enumerate}
\end{example}

\begin{defn}
Let \((D,S,B)\) be a datum.
We define a ringed space, called the \emph{relative Matsui spectrum} of the datum \((D,S,B)\), by
\begin{align*}
  \Spec(D,S,B)&:=(\Spc(D,S),\sO),\\
  \Spc(D,S)&:=\Spc(D,\Sub_S(D)), \quad
  \sO:=a\Bigl(U \mapsto \Nat_B^0(\Id_{D_U},\Id_{D_U}) \Bigr),
\end{align*}
where \(a\) denotes sheafification, and \(D_U\) is naturally regarded as an \(S\)-module.
\end{defn}

\begin{rem}
If \(S=B\), then one can show that \(\Spec(D,S,B)\) is a locally ringed space, although we omit the proof here.
By Theorem~\ref{thm:restriction_of_the_base}, many spaces \(\Spec(D,S,B)\) embed into \(\Spec(D,B,B)\); in this case, \(\Spec(D,S,B)\) is also a locally ringed space.
Moreover, if \(\Spec(D,S,B)\) is a locally ringed space and admissible, then the morphism in Theorem~\ref{thm:morphism_to_the_base} is a morphism of locally ringed spaces.
\end{rem}

\section{Universality}\label{section:universality}
In this section, we give a universal criterion for our spaces.
Matsui~\cite{Ma1} provided a characterization of the Matsui spectrum, but our result is slightly different and more general.

\begin{defn}
  Let \((D,S)\) be a datum.
  A \emph{support datum} of \((D,S)\) is a topological \(T_0\)-space \(X\) equipped with a map \(\Supp:\mo{Ob}(D)\to 2^X\) satisfying the following conditions:
  \begin{enumerate}
    \item (shift invariance) \(\Supp(a)=\Supp(a[1])\) for any \(a\in D\).
    \item (action) For any \(a\in D\), \(s\in S\), we have \(\Supp(s\otimes a)\subset \Supp(a)\).
    \item (direct sums) \(\Supp(a\oplus b)=\Supp(a)\cup \Supp(b)\) for any \(a,b\in D\).
    \item (exactness) \(\Supp(b)\subset \Supp(a)\cup \Supp(c)\) for any exact triangle \(a\to b\to c\) in \(D\).
    \item (distinguishability) For any two distinct submodules \(I,J\subset D\), we have \(\bigcup_{a\in I}\Supp(a)\neq \bigcup_{b\in J}\Supp(b)\).
    \item (topology) The family of closed subsets of \(X\) is generated by \(\{\Supp(a)\mid a\in D\}\).
  \end{enumerate}
  Let \(X,Y\) be support data of \((D,S)\).
  A \emph{map of support data} is a continuous map \(f:X\to Y\) compatible with \(\Supp\), i.e., \(f^{-1}(\Supp_Y(a))=\Supp_X(a)\) for all \(a\in D\).

  A support datum \(X\) of \((D,S)\) is called \emph{universal} if for any support datum \(Y\), there exists a unique map \(X\to Y\) of support data.
\end{defn}

\begin{thm}\label{thm:universality_of_matsui_spectra}
  Let \((D,S)\) be a datum.
  Then the relative Matsui spectrum \(\Spc (D,S)\) is the universal support datum.
\end{thm}

For a triangulated category \(T\), a topological \(T_0\)-space \(X\) with an assignment \(\Supp:T\to X\) is called a support datum of \(T\) if it satisfies the above conditions except (2).
Our conditions are stronger than Matsui's original definition~\cite{Ma1}, where (5), (6), and the \(T_0\)-condition are not imposed.

\begin{thm}
Let \(T\) be a triangulated category.
Then the Matsui spectrum \(\tSpc(T)\) is a universal support datum of \(T\).
\end{thm}
The proof is the same as that of Theorem~\ref{thm:universality_of_matsui_spectra}.

\medskip

We now prove this fact in a more general setting of semilattices.

A \emph{semilattice} is a partially ordered set such that any two elements admit a supremum.
A map of semilattices is a map preserving finite suprema.
A semilattice is called \emph{ind-complete} if any non-empty subset admits a supremum.

Let \(L\) be a semilattice.
A subset \(S\subset L\) is called an \emph{ind-object} if it is closed under taking suprema of two elements, is non-empty, and satisfies: if \(s\in S\) and \(l\le s\) then \(l\in S\).
The \emph{ind-completion} \(\Ind(L)\) is the set of all ind-objects.
It is partially ordered by inclusion and forms an ind-complete semilattice.
For any ind-complete semilattice \(L'\) and a semilattice map \(f:L\to L'\), there exists a unique extension \(\tilde{f}:\Ind(L)\to L'\) preserving arbitrary suprema.
A preordered set may naturally be regarded as a category.
It is easy to check that \(\Ind(L)\), together with the natural morphism \(L\to \Ind(L): l\mapsto \{l'\in L\mid l'\le l\}\), coincides with the categorical ind-completion.

For a set \(X\), \(2^X\) is partially ordered by inclusion, and is an ind-complete semilattice.

\begin{defn}
  Let \(L\) be a semilattice.
  A \emph{support datum} of \(L\) is a topological \(T_0\)-space \(X\) with a map of semilattices \(\Supp:L\to 2^X\) satisfying:
  \begin{enumerate}
    \item (distinguishability) The induced map \(\Supp:\Ind(L)\to 2^X\) is injective.
    \item (topology) The family of closed subsets of \(X\) is generated by \(\{\Supp(l)\mid l\in L\}\).
  \end{enumerate}
  Let \(X,Y\) be support data of \(L\).
  A \emph{map of support data} is a continuous map \(f:X\to Y\) such that \(\Supp_X(l)=f^{-1}(\Supp_Y(l))\) for any \(l\in L\).

  A support datum \(X\) of \(L\) is called \emph{universal} if for any support datum \(Y\) of \(L\), there exists a unique map \(X\to Y\) of support data.
\end{defn}

Let \(L\) be a semilattice.
We set
\begin{align*}
X_L&=\{P\in \Ind(L)\mid \bigcap_{S\in \Ind(L), S\supsetneq P}S\supsetneq P\},\\
\Supp(l)&=\{S\in \Ind(L)\mid l\notin S\}\quad (l\in L).
\end{align*}
We topologize \(X_L\) by declaring the family of closed subsets to be generated by \(\{\Supp(l)\mid l\in L\}\).

\begin{thm}
  \(X_L\) is the universal support datum of \(L\).
\end{thm}
\begin{proof}
  First, we prove that \(X_L\) is a support datum.  
Take \(S\in \Ind(L)\) and \(l\in L\setminus S\).  
The set \(\{T\in \Ind(L)\mid S\subset T,\, l\notin T\}\) has a maximal element \(P\) by Zorn’s lemma.  
Since \(l\in \bigcap_{T\in \Ind(L),\, T\supsetneq P}T\), it follows that \(P\in X_L\).  
This establishes the injectivity of \(\Supp:\Ind(L)\to 2^{X_L}\).  
The remaining conditions are straightforward.

  Next, we prove universality.
  Let \(Y=(Y,\Supp_Y)\) be another support datum.
  For \(P\in X_L\), let \(\ol{P}=\bigcap_{S\in \Ind(L),\, S\supsetneq P}S\).
  By definition, \(\Supp_Y(\ol{P})\setminus \Supp_Y(P)\neq \emptyset\).
  Choose a point \(y_P\in \Supp_Y(\ol{P})\setminus \Supp_Y(P)\).
  Define \(f:X_L\to Y\) by \(P\mapsto y_P\).
  If \(l\notin P\), then \(\Supp_Y(l)\cup\Supp_Y(P)=\Supp_Y(\sup\{P,l\})=\Supp_Y(\ol{P})\).
  Thus \(f^{-1}(\Supp_Y(l))=\{P\in X_L\mid l\notin P\}\), hence \(f\) is a map of support data.
  Since any map of support data is compatible with the natural injection \(Y\to 2^L: y\mapsto \{l\in L\mid y\notin \Supp_Y(l)\}\), uniqueness follows.
\end{proof}

In fact, isomorphism classes of support data of \(L\) correspond bijectively to subsets of \(\Ind(L)\cup\{\emptyset\}\) containing \(X_L\).

\begin{proof}[Proof of Theorem \ref{thm:universality_of_matsui_spectra}]
Let \(\Sub_{S}^{\mo{fin}}(D)\) be the semilattice of finitely generated submodules of \(D\) ordered by inclusion.
Then \(\Ind(\Sub_{S}^{\mo{fin}}(D))\) is isomorphic to the semilattice \(\Sub_{S}(D)\), and \(X_{\Sub_{S}^{\mo{fin}}(D)}\) is the set of prime thick subcategories.
Thus \(\Spc(D,S)\) is the universal support datum of \(\Sub_{S}^{\mo{fin}}(D)\).

For a support datum \(Y\) of \((D,S)\), the map
\(\Sub_{S}^{\mo{fin}}(D)\to 2^Y: I\mapsto \bigcup_{a\in I}\Supp(a)\) makes \(Y\) a support datum of \(\Sub_{S}^{\mo{fin}}(D)\).
Conversely, for a support datum \(Y'\) of \(\Sub_{S}^{\mo{fin}}(D)\), the map \(\mo{Ob}(D)\to 2^{Y'}: a\mapsto \Supp(\ang{a})\) makes \(Y'\) a support datum of \((D,S)\).
These constructions are mutually inverse, so the categories of support data of \((D,S)\) and of \(\Sub_{S}^{\mo{fin}}(D)\) are equivalent.
This proves that \(\Spc(D,S)\) is the universal support datum of \((D,S)\).
\end{proof} 

\begin{example}
For a tensor triangulated category \(T\), let \(\mo{Ideal}_{\otimes,fin}(T)\) be the semilattice of radicals of finitely generated tensor ideals of \(T\).
Then the Balmer spectrum \(\oSpc T\) is a support datum of \(\mo{Ideal}_{\otimes,fin}(T)\), but in general it is not universal in our sense.
If \(T\) is rigid, then \(\mo{Ideal}_{\otimes,fin}(T)\) coincides with the semilattice of finitely generated submodules of \(T\), and by universality there is a natural embedding \(\Spc(T,T)\to \oSpc T\).
\end{example}

\section{Fundamental Results}\label{section:fundamental_results}
In this section, we present several fundamental results on relative Matsui spectra, including the classification of submodules and comparisons with the Balmer and Matsui spectra.
The classification result in Theorem~\ref{thm:classification_of_submodules}, together with its corollaries, can also be established for the original Matsui spectra in the same manner.

\begin{prop}\label{prop:open_embedding}
Let \((D,S)\) be a datum, and let \(I\subset D\) be an \(S\)-submodule.
Then there is a natural morphism
\[
\Spec(\Idem(D/I),S)\longrightarrow \Spec(D,S).
\]
If \(I\) is finitely generated, this is an open embedding.
\end{prop}
\begin{proof}
The proof is the same as that of \cite[Propositions~4.2 and 4.3]{Ma2}.
\end{proof}

\begin{prop}
  Let \((D,S)\) be a datum, and let \(I\subset S\) be a two-sided ideal such that \(I\otimes D=0\).
  Then there is a natural isomorphism
  \[
  \Spec(D,\Idem(S/I))\cong \Spec(D,S).
  \]
\end{prop}
\begin{proof}
We have \(\Sub_S(D)=\Sub_{\Idem(S/I)}(D)\), and the isomorphism follows from the definition.
\end{proof}

\begin{thm}[Classification of submodules]\label{thm:classification_of_submodules} Let \((D,S)\) be a datum. Then, the map 
  \[\Supp:\Sub_S(D)\to \{\text{subsets of }\mo{\Spc(D,S)}\}\]
  is injective and its left inverse is given by
  \[\Spc(D,S)\supset Z\mapsto \{a\in D\mid \Supp(a)\subset Z\}.\]
\end{thm}
  \begin{proof} See Lemma \ref{lem:classification_of_submodules}. \end{proof}

\begin{cor}
Every \(S\)-submodule of \(D\) is an intersection of prime \(S\)-submodules.
\end{cor}
\begin{proof}
By definition, for any subset \(Z \subset \Spc(D,S)\),
\[
\{\,a \in D \mid \Supp(a) \subset Z \,\}
= \bigcap_{P \in \Spc(D,S)\setminus Z} P.
\]
\end{proof}

\begin{cor}\label{cor:empty_spectrum}
\(\Spc(D,S)=\emptyset\) if and only if \(D\cong 0\).
\end{cor}
\begin{proof}
  If \(D\not\cong 0\), then \(D\) has at least two distinct \(S\)-submodules: \(0\) and \(D\).
  By Theorem~\ref{thm:classification_of_submodules}, it follows that \(\Spc(D,S)\neq \emptyset\).
\end{proof}

\begin{rem}
Matsui and Takahashi \cite[Theorem 4.5]{Ma1}, \cite[Theorem 2.9]{MT} proved that the Matsui spectrum \(\tSpc T\), together with the support map, classifies radical thick subcategories of a triangulated category \(T\).  
In the same manner as above, we can show that every thick subcategory is radical in the sense of \cite{Ma1}, and in particular that the Matsui spectrum classifies all thick subcategories.  
Moreover, we can show that the Matsui spectrum \(\tSpc T\) is empty if and only if \(T \cong 0\).  
This was proved in \cite[Proposition~2.17]{HKO} in the case where \(T\) is classically finitely generated, that is, finitely generated as a thick subcategory of itself.
\end{rem}

For a commutative algebra object \(D\) in \(\Perf\), the Balmer spectrum \(\oSpc D\) classifies radical tensor ideals of \(D\) in the sense of \cite{Ba}, whereas our space \(\Spc(D,D)\) classifies all tensor ideals.  
If \(D\) is rigid, then every tensor ideal is radical, and hence our space is closely related to the Balmer spectrum.

\begin{thm}[Comparison with the Balmer spectrum]\label{thm:comparison_with_the_Balmer_spectrum}
Let \(D\) be a commutative algebra object in \(\Perf\), and let \(D_0\) denote its homotopy category.
Assume \(D\) is rigid.
Then there is a natural dense embedding
\[
\Spec(D,D,D)\longrightarrow \oSpec(D_0).
\]
If \(\oSpec(D_0)\) is topologically noetherian, then this embedding is an isomorphism.
\end{thm}
\begin{proof}
  Tensor ideals of \(D_0\) correspond bijectively to submodules of \(D\).
  By \cite[Remark~4.3]{Ba}, any tensor ideal is radical.
  Hence the underlying space \(\oSpc(D_0)\) is \(\Spc'(D,\Sub_D(D))\).
  There is a natural dense embedding \(\iota:\Spc(D,\Sub_D(D))\to\Spc'(D,\Sub_D(D))\).
  The morphism of structure sheaves is induced by the isomorphisms
  \[
  \Nat_D^0(\Id_{D_{\iota^{-1}(U)}},\Id_{D_{\iota^{-1}(U)}})\cong \End(1_{D_{0,U}}),\quad U\subset \oSpec(D_0),
  \]
  where \(1_{D_{0,U}}\) denotes the unit object of \(D_{0,U}\).
  Note that \(D_{0,U}\) is the homotopy category of \(D_{\iota^{-1}(U)}\), and any \(\eta\in \Nat_D^0(\Id_{D_{\iota^{-1}(U)}},\Id_{D_{\iota^{-1}(U)}})\) is determined by its evaluation at \(1_{D_{\iota^{-1}(U)}}\), since every object of \(D_{\iota^{-1}(U)}\) is equivalent to a retract of \(d\otimes 1_{D_{\iota^{-1}(U)}}\) for some \(d\in D\).

  Assume now that \(\oSpec(D_0)\) is topologically noetherian.
  Let \(P\) be a quasi-\(\Sub_D(D)\)-prime, and let \(x\in\oSpec(D_0)\) be the corresponding point.
  To prove that \(\iota\) is surjective, we must show that \(P\) is in fact an \(\Sub_D(D)\)-prime.
  We have \(\Supp_\otimes(P)=\{y\in\oSpec(D_0)\mid x\notin\overline{\{y\}}\}\).
  Let \(\overline{P}\) be the tensor ideal with support \(\Supp_\otimes(P)\cup \{x\}\).
  By \cite[Theorem~4.10]{Ba}, such \(\overline{P}\) exists and is the smallest tensor ideal strictly containing \(P\).
  Thus \(P\) is a prime submodule.
\end{proof}

\begin{defn}[Matsui spectrum]
Let \(k\) be an \(E_\infty\)-ring, and let \(D\) be a \(k\)-linear stable \(\infty\)-category (i.e., a \(\perf(k)\)-module).
The ringed space
\begin{align*}
  \tSpec(D)&:=(\tSpc(D),\sO),\\
  \tSpc(D)&:=\Spc(D,\{\text{thick subcategories of }D\}),\\
  \sO&:=a\bigl(U\mapsto \Nat^0_k(\Id_{D_U},\Id_{D_U})\bigr)
\end{align*}
is called the \emph{Matsui spectrum} of \(D\).
\end{defn}
Note that there is a natural homeomorphic embedding \(\tSpec(D_0)\to \tSpec(D)\) from the original Matsui spectrum of the homotopy category \(D_0\) to our spectrum.

\begin{thm}[Comparison with the Matsui spectrum]\label{thm:comparison_with_the_Matsui_spectrum}
Let \(k\) be an \(E_\infty\)-ring and \((D,S,\perf(k))\) a datum.
Assume \(S\) is classically generated by \(1_S\).
Then there is an isomorphism
\[
\Spec(D,S,\perf(k))\cong \tSpec(D).
\]
\end{thm}
\begin{proof}
By assumption, all thick subcategories of \(D\) are \(S\)-submodules.
Hence their underlying spaces coincide, and so do the structure sheaves by definition.
\end{proof}

\section{Commutativity of Localizations}\label{section:commutativity_of_localizations}
In this section, we recall some fundamental properties concerning the commutativity of localizations.
These are well-known facts, so we omit the proofs; see, for example, \cite{Kr} or \cite[Proposition~10.5]{Sc} for details.

Let \(D\) be a presentable stable \(\infty\)-category.
A \emph{presentable thick subcategory} \(I\) of \(D\) is a thick subcategory such that \(I\) itself is presentable and the inclusion functor preserves small colimits.
An object \(d\in D\) is called \emph{\(I\)-local} if \(\Hom^*(I,d)=0\).
The Verdier quotient \(D/I\) is a presentable stable \(\infty\)-category, and the localization functor \(D \to D/I\), denoted by \(L_I\), has a fully faithful right adjoint.
Thus, we may identify \(D/I\) with the full subcategory of \(I\)-local objects in \(D\).

\begin{lem}\label{lem:composability_of_large_localization}
Let \(D\) be a presentable stable \(\infty\)-category, and let \(I,J\) be presentable thick subcategories.
Then the following conditions are equivalent:
\begin{enumerate}
\item For any \(x\in I\) and \(y\in J\), every morphism \(y\to x\) factors through some \(z\in I\cap J\).
\item \(L_I\circ L_J= L_{\ang{I,J}}\).
\item \(L_I\) preserves \(J\)-local objects.
\item \(L_J\) preserves \(I\).
\end{enumerate}
\end{lem}

\begin{lem}
Let \(D\) be a presentable stable \(\infty\)-category and \(I,J\) its presentable submodules.
Then, the following conditions are equivalent.
\begin{enumerate}
\item Both \((I,J)\) and \((J,I)\) satisfy the equivalent conditions of Lemma~\ref{lem:composability_of_large_localization}.
\item The following diagram is Cartesian:
\[\xymatrix{
D/(I\cap J)\ar[r]\ar[d]&  D/J\ar[d]\\
D/I\ar[r] & D/\ang{I,J}^{\mo{pr}}.
}\]
\item \(L_I\circ L_J= L_J\circ L_I\).
\end{enumerate}
\end{lem}

Let \(D\) be a (small) stable \(\infty\)-category.
The ind-completion is denoted by \(\Ind(D)\).
For a thick subcategory \(I\subset D\), the ind-completion \(\Ind(I)\) is a thick subcategory of \(\Ind(D)\), and there is a natural equivalence \(\Ind(D/I)\cong \Ind(D)/\Ind(I)\).

\begin{lem}\label{lem:composability_of_small_localization}
Let \(D\) be a stable \(\infty\)-category, and let \(I,J\subset D\) be thick subcategories.
Then the following conditions are equivalent:
\begin{enumerate}
\item For any \(x\in I\) and \(y\in J\), every morphism \(y\to x\) factors through some \(z\in I\cap J\).
\item The pair of presentable submodules \((\Ind~ I,\Ind~ J)\subset\Ind ~D\) satisfies the equivalent conditions of Lemma~\ref{lem:composability_of_large_localization}, and moreover \(\Ind ~I\cap \Ind ~J=\Ind(I\cap J)\).
\end{enumerate}
\end{lem}

\begin{lem}\label{lem:commutativity_of_small_localization}
Let \(D\) be a stable \(\infty\)-category, and let \(I,J\subset D\) be thick subcategories.
Then the following conditions are equivalent:
\begin{enumerate}
\item Both \((I,J)\) and \((J,I)\) satisfy the equivalent conditions of Lemma~\ref{lem:composability_of_small_localization}.
\item The following diagram is Cartesian:
\[
\xymatrix{
\Idem(D/(I\cap J))\ar[d]\ar[r]&\Idem(D/J)\ar[d]\\
\Idem(D/I)\ar[r]&\Idem(D/\ang{I,J}).
}
\]
\end{enumerate}
\end{lem}

\begin{defn}
Let \(D\) be a stable \(\infty\)-category, and let \(I,J\subset D\) be thick subcategories.
We say that \((I,J)\) is \emph{composable} (resp. \emph{commutative}) if it satisfies the equivalent conditions of Lemma~\ref{lem:composability_of_small_localization} (resp.~Lemma~\ref{lem:commutativity_of_small_localization}).
\end{defn}

\begin{lem}\label{lem:submodule_reconstruction}
Let \(D\) be a stable \(\infty\)-category, and let \(I,J,K\) be thick subcategories such that \((I,J)\) and \((I,K)\) are composable, and \(\Ind J\cap \Ind K=\Ind (J\cap K)\).  
Then \(\ang{I,J}\cap \ang{I,K}=\ang{I,J\cap K}\).
\end{lem}
\begin{proof}
The inclusion \(\ang{I,J\cap K}\subset \ang{I,J}\cap \ang{I,K}\) is clear.  
Let \(F\in (\ang{I,J}\cap \ang{I,K})/\ang{I,J\cap K}\subset \Ind D\).  
By composability, we have \(F\in \Ind J\cap \Ind K=\Ind(J\cap K)\),  
and hence \((\ang{I,J}\cap \ang{I,K})/\ang{I,J\cap K}=0\).
\end{proof}

\begin{prop}\label{prop:cartesian_diagram_of_submodules}
Let \((D,S)\) be a datum, and let \(I,J\) be submodules such that, for any submodule \(K\), both pairs \((K,I)\) and \((K,J)\) are composable.  
Then the following diagram is Cartesian:
\[
  \xymatrix{
    \Sub_S(\Idem(D/(I\cap J)))\ar[d]\ar[r]&\Sub_S(\Idem(D/J))\ar[d]\\
    \Sub_S(\Idem(D/I))\ar[r]&\Sub_S(\Idem(D/\ang{I,J})).
  }
\]
\end{prop}
\begin{proof}
Submodules of \(\Idem(D/I)\) correspond bijectively to submodules of \(D\) containing \(I\), and similarly for the other terms.  
Let \(P\) be the pullback of the above diagram.  
There is a natural map \(\Sub_S(\Idem(D/(I\cap J)))\to P\), which is injective by Lemma~\ref{lem:submodule_reconstruction}.  
Now take \(K\in \Sub_S(\Idem(D/I))\) and \(K'\in \Sub_S(\Idem(D/J))\) with the same image \(K''\) in \(\Sub_S(\Idem(D/\ang{I,J}))\).  
Then, by Lemma~\ref{lem:commutativity_of_small_localization}, the pullback \(K'''\) of \(K\to K''\gets K'\) is a submodule of \(\Idem(D/(I\cap J))\).  
For any \(k\in K\), there exists \(k'\in K'\) such that \(k\oplus k[-1]\) and \(k'\) have the same image in \(K''\) by \cite[Theorem~2.1]{Th}.  
This shows that the image of \(K'''\) in \(\Sub_S(\Idem(D/I))\) is \(K\), and similarly its image in \(\Sub_S(\Idem(D/J))\) is \(K'\).  
Hence surjectivity holds.
\end{proof}

Let \((D,S)\) be a datum and \(I\) an \(S\)-submodule of \(D\).
Let \(\Spc(D,S)_I\) denote the subspace of \(\Spc(D,S)\) consisting of primes \(P\) such that \(I\not\subset P\).

\begin{prop}\label{prop:composability_and_spaces}
Under the above assumptions, suppose that for any submodule \(J\) of \(D\), the pair \((J,I)\) is composable.
Then there is a bijection between the underlying sets:
\begin{align*}
\Spc(I,S)&\longrightarrow \Spc(D,S)_I,\\
\tilde{P}&\longmapsto \{a\in D\mid \Fib(a\to L_I(a))\in \Ind \tilde{P}\},\\
P\cap I&\mapsfrom P.
\end{align*}
The complement of \(\Spc(D,S)_I\) in \(\Spc(D,S)\) is \(\Spc(\Idem(D/I),S)\).
\end{prop}
\begin{proof}
The last assertion follows directly from the definition.
Let \(\iota:I\hookrightarrow D\) denote the natural inclusion.
Define a map
\[
\phi=(\iota^{-1},\ang{L_I}):\Sub_S(D)\longrightarrow \Sub_S(I)\times \Sub_S(\Idem(D/I)),\quad
J\longmapsto (\iota^{-1}(J), \Idem(\ang{J,I}_S/I)).
\]
Let \(J,K\subset D\) be submodules with \(J\subset K\).
We claim that \(J=K\) if and only if \(\phi(J)=\phi(K)\).
The “only if” direction is clear.
For the “if” direction, suppose \(\phi(J)=\phi(K)\).
Since the pair \((J, I)\) is composable, an object \(a \in \Ind(J)\) is \(\Ind(I)\)-local if and only if it is \(\Ind(J \cap I)\)-local.  
Therefore, in \(\Ind(D)\), the subcategory \(\Idem(\ang{J, I}_S/I)\) is identified with \(\Idem(J/(J \cap I))\), and similarly for \(K\).
The submodules of \(D\) containing \(J \cap I\) correspond bijectively to the submodules of \(\Idem(D/(J \cap I))\).  
Thus, the equalities \(J \cap I = K \cap I\) and \(\Idem(J/(J \cap I)) = \Idem(K/(K \cap I))\) together imply \(J = K\).

Next, let \(J\subset D\) and \(K\subset I\) be submodules such that \(J \cap I\subset K\).
We claim that \(\phi(\ang{J\cap K}_S)=(K,\Idem(\ang{J,I}_S/I))\).
The second component is clear.
The inclusion \(\ang{J,K}_S\cap I\supset K\) is obvious.
Define the functor \(\psi:D\to \Ind(I)\), \(a\mapsto \Fib(a\to L_I(a))\).
By composability, we have \(\psi(J)\subset\Ind(J\cap I)\subset \Ind(K)\).
Thus,
\[
\ang{J,K}_{S}\cap I
  \subset \psi(\ang{J,K}_{S})
  \subset \ang{\psi(J)\cup\psi(K)}_{S}
  = \ang{\psi(J)\cup K}_{S}
  \subset \Ind(K),
\]
which implies \(\ang{J,K}_S\cap I\subset \Ind(K)\cap I=K\).

Note that for a submodule \(J\subset D\), we have \(\ang{\psi(J)}^{\mo{pr}}_S=\Ind(J\cap I)\) by composability.  
Combining the above arguments with elementary set-theoretic considerations, we see that a submodule \(J\subset D\) with \(I\not\subset J\) is prime if and only if \(J\cap I\) is prime in \(I\) and \(J\) is maximal among submodules \(K\) satisfying \(K\cap I=J\cap I\).  
This maximality is equivalent to the condition \(J=\{a\in D \mid \psi(a)\in \Ind(J\cap I)\}\).
Hence, the correspondences in the statement are well defined and mutually inverse.
\end{proof}

\section{The Morphism to the Spectrum of the Base}\label{section:morphism_to_base}
In this section, we construct a morphism from the relative Matsui spectrum of a given datum to the spectrum of the base.

Let \((D,S)\) be a datum.
For a specialization-closed subset \(Z\subset \Spec(D,S)\), set
\[
S_Z=\{s\in S \mid \Supp(s)\subset Z\}, 
\quad D(Z)=\ang{S_Z\otimes D}.
\]

\begin{defn}
  \begin{enumerate}
    \item A datum \((D,S)\) is called \emph{admissible} if it satisfies the following conditions:
    \begin{enumerate}
      \item \(\Spec(S,S)\) is topologically noetherian.
      \item For any closed subset \(Z\subset \Spec(S,S)\) and any \(S\)-submodule \(I\subset D\), the pair \((I,D(Z))\) is composable.
      \item For any closed subsets \(Z,W\subset \Spec(S,S)\), we have \(D(Z\cap W)=D(Z)\cap D(W)\).
      \item For any family of closed subsets \(\{Z_i\}_i\subset \Spec(S,S)\), we have \(D(\bigcup_i Z_i)=\ang{\bigcup_i D(Z_i)}\).
    \end{enumerate}

    \item A datum \((D,S)\) is called \emph{smashing} if \(S\) is a commutative algebra, rigid, \(\oSpec(D_0)\) is topologically noetherian, and for any closed subset \(Z\subset \Spec(S,S)\) and any \(F\in D\), the natural morphism \(L_{S_Z}(1_S)\otimes F\to L_{D(Z)}(F)\) is an isomorphism.

    \item A datum \((D,S)\) is called \emph{rigid} if \(S\) is a commutative algebra, rigid, \(\oSpec(D_0)\) is topologically noetherian, and there exists a natural equivalence \(\Hom^0(s\otimes d, d')\cong \Hom^0(d,s^\vee \otimes d' )\) for all \(s\in S,\, d,d'\in D\), where \(s^\vee\) denotes the dual of \(s\).
  \end{enumerate}
\end{defn}

We will prove below that
\[
\textbf{rigid} \;\Rightarrow\; \textbf{smashing} \;\Rightarrow\; \textbf{admissible}.
\]

Let \((D,S)\) be a datum.
Define \(\fSpec(D,S)\) to be the ringed space with underlying set \(\Spec(D,S)\), whose family of closed subsets is generated by
\[
\{\Supp(s\otimes D)\mid s\in S\} \cup \{\Supp(d)\mid d\in D\}.
\]
The structure sheaf is defined in the same way as for \(\Spec(D,S)\).
There is a natural bijective morphism \(\fSpec(D,S)\to\Spec(D,S)\).
If \(D\) is classically generated by a single object \(d\), then \(\Supp(s\otimes D)=\Supp(s\otimes d)\), and hence this morphism is an isomorphism.

\begin{thm}[The morphism to the spectrum of the base]\label{thm:morphism_to_the_base}
Let \((D,S,B)\) be an admissible datum.
Then there exists a natural morphism
\[
\fSpec(D,S,B)\to \Spec(S,S,S), \quad 
P\longmapsto \{s\in S\mid s\otimes D\subset P\}.
\]
\end{thm}
\begin{proof}
  By Lemma~\ref{lem:base_prime_submodule}, the submodule \(\{s\in S\mid s\otimes D\subset P\}\subset S\) is prime, hence the morphism is well-defined.
  The continuity follows because the inverse image of closed sets of the form \(\Supp(s)\) is \(\Supp(s\otimes D)\).
  Let \(U\subset \Spec(S,S)\) be an open subset, and let \(V\subset \fSpec(D,S)\) be its inverse image.
  The induced morphism of structure sheaves is given by
  \[
  \Nat_S^0(\Id_{S_U},\Id_{S_U})
  \xrightarrow{\;\mo{ev}_{1_{S_U}}\;} \End^0(1_{S_U})
  \xrightarrow{\;1_{S_U}\otimes -\;} \Nat_B^0(\Id_{D_V},\Id_{D_V}),\quad
  \eta\mapsto \eta_{1_{S_U}}\otimes -.
  \]
\end{proof}

\begin{lem}\label{lem:base_prime_submodule}
  For \(P\in \Spec(D,S)\), the set \(\{s\in S\mid s\otimes D\subset P\}\) is a prime submodule of \(S\).
\end{lem}
\begin{proof}
  Put \(Y=\Spec(S,S)\).  
A specialization-closed subset \(Z\subset Y\) is said to satisfy condition (*) if \(D(Z)\not\subset P\).  
Let \(Z,W\subset Y\) be closed subsets satisfying (*).  
By Lemma~\ref{lem:submodule_reconstruction}, we have \(\ang{P,D(Z\cap W)}=\ang{P,D(Z)}\cap \ang{P,D(W)} \supsetneq P\), so \(Z\cap W\) also satisfies (*).

  Let \(\{Z_i\}_i\) be a family of specialization-closed subsets of \(Y\), and put \(Z=\bigcup_i Z_i\).
  Then \(Z\) satisfies (*) if and only if at least one \(Z_i\) does.
  The “if” part is immediate.
  For the “only if” part, use the hypothesis \(D(Z)=\ang{\bigcup_i D(Z_i)}\).

  Define \(W=\{Z\subset Y \mid Z \text{ closed and satisfies (*)}\}\).
  This set is non-empty since it contains \(Y\), and it is closed under finite intersections by the previous argument.
  As \(Y\) is noetherian, \(W\) has a smallest element \(Z_0\).
  We have the equality \(Z_0=\bigcup_{z\in Z_0}\ol{\{z\}}\).
  By the previous claim, there exists \(z_0\in Z_0\) such that \(\ol{\{z_0\}}\) satisfies (*).
  By minimality of \(Z_0\), it follows that \(Z_0=\ol{\{z_0\}}\).
  Since \(Y\) is a \(T_0\)-space, such a point \(z_0\) is unique.

  A specialization-closed subset \(Z\subset Y\) satisfies (*) if and only if it contains \(Z_0\), equivalently if it contains \(z_0\).
  Therefore,
  \[
  \{s\in S\mid s\otimes D\subset P\}
  =\{s\in S\mid z_0\not\in \Supp(s)\},
  \]
  which shows that the submodule is the prime corresponding to \(z_0\).
\end{proof}

\begin{cor}
Let \((D,S)\) be a smashing datum.
Then there is a natural morphism 
\[
\fSpec(D,S)\to \oSpec(S).
\]
\end{cor}
\begin{proof}
We have an isomorphism \(\oSpec(S)\cong \Spec(S,S,S)\).
\end{proof}

\begin{thm}\label{thm:local_decomposition}
Let \((D,S)\) be an admissible datum.
For specialization-closed subsets \(Z\subset W\subset \Spec(S,S)\), put \(D_{W\setminus Z}=\Idem(\ang{S_W\otimes D}/\ang{S_Z \otimes D})\).
Then there is a natural continuous bijection
\[
\Spc(D,S)_{W\setminus Z}\to \Spc(D_{W\setminus Z},S),
\]
where \(\Spc(D,S)_{W\setminus Z}\) denotes the set-theoretic fiber with the induced topology.
If \((D,S)\) is smashing and \(W\) is closed, then this map is a homeomorphism.
\end{thm}
\begin{proof}
The bijection follows from Proposition~\ref{prop:composability_and_spaces}.
Assume \((D,S)\) is smashing and \(W\) is closed.
By \cite[Proposition~2.14]{Ba}, there exists \(s\in S\) such that \(\ang{s}_S=S_W\).
For any \(a\in D\), we have
\(\Supp(a)\cap \Spc(D,S)_{W\setminus Z}=\Supp(s\otimes a)\).
Thus the map is closed, and hence a homeomorphism.
\end{proof}

\begin{thm}
Let \((D,S)\) be an admissible datum.  
Then the assignment
\[
\Spc(S,S)\supset U\ \longmapsto\ \Sub_S(D_U)
\]
defines a sheaf on \(\Spc(S,S)\).
\end{thm}

\begin{proof}
The claim follows directly from Proposition~\ref{prop:cartesian_diagram_of_submodules}.
\end{proof}

\begin{thm}[Restriction of the base]\label{thm:restriction_of_the_base}
Let \(B,S',S\) be algebra objects of \(\Perf\) with homomorphisms \(B\to S'\) and \(S'\to S\), and let \(D\) be a left \(S\)-module.
Assume \(S\) is commutative, \((D,S,B)\) is admissible, and \(S\) is \(\Spec(S,S)\)-locally classically generated by the image of \(S'\).
Then there is an embedding 
\[
\Spec(D,S)\to \Spec(D,S').
\]
If \(S\) itself is classically generated by the image of \(S'\), then the morphism is an isomorphism.
If \(D\) is a finitely generated \(S'\)-module, then the morphism is an open embedding.
\end{thm}
\begin{proof}
  First assume \(S\) is classically generated by the image of \(S'\).
  Then \(\Sub_S(D)=\Sub_{S'}(D)\), hence \(\Spec(D,S)\cong \Spec(D,S')\).

  In general, let \(\{U_i\}_i\) be an open covering of \(\Spec(S,S)\) such that \(S_{U_i}\) is classically generated by the image of \(S'\).
  Then
  \(\Spec(D_{U_i},S)\cong \Spec(D_{U_i},S_{U_i})\cong \Spec(D_{U_i},S')\),
  and we have embeddings \(\Spec(D_{U_i},S')\to \Spec(D,S')\).
  Since \(\Spec(D_{U_i},S)\) covers \(\Spec(D,S)\), these local maps glue to a global morphism \(\Spec(D,S)\to \Spec(D,S')\).
  As their closed subsets are generated by \(\{\Supp(a)\mid a\in D\}\), this map is continuous and is an embedding.

  Now assume \(D\) is a finitely generated \(S'\)-module.
  Let \(d_1,\ldots,d_n\) be generators and put \(d=d_1\oplus \cdots \oplus d_n\).
  Since the sets \(\Supp(s)\) for \(s\in S\) form a basis of closed subsets of \(\Spec(S,S)\), we may assume that \(U_i=\Spec(S,S)\setminus \Supp(s_i)\) with \(s_i\in S\).
  It suffices to prove that the embedding \(\Spec(D_{U_i},S')\to \Spec(D,S')\) is open.
  For any \(s\in S\),
  \[
  s\otimes s_i\otimes d
  \cong s_i\otimes s\otimes d
  \subset s_i\otimes s\otimes \ang{d}_{S'}
  =s_i\otimes s\otimes D
  \subset s_i\otimes D
  \subset \ang{s_i\otimes D}_{S'}.
  \]
  Therefore, \(\ang{I_{U_i}\otimes D}_{S}=\ang{s\otimes d}_{S'}\) is a finitely generated \(S'\)-module, and the embedding is open.
\end{proof}

\begin{prop}
Let \((D,S)\) be a datum.
\begin{enumerate}
\item If \((D,S)\) is rigid, then it is smashing.
\item If \((D,S)\) is smashing, then it is admissible.
\end{enumerate}
\end{prop}
\begin{proof}
(1) follows from \cite[Proposition~9]{Miller1992finite}.
For (2), we first prove the composability condition.
Let \(x\in D\), let \(Z\subset \Spec(S,S)\) be closed, \(S_Z\) the corresponding submodule of \(S\), and \(y\in D(Z)\).
Note that \((S,S)\) is smashing by (1).
Let \(1_Z=\Fib(1_S\to L_{S_Z}(1_S))\).
The object \(1_Z\) can be written as a filtered colimit \(\mathrm{colim}\, z_i\) with \(z_i\in S_Z\).
Take any morphism \(f:y\to x\).
Since the composite \(y\to x\to L_{D(Z)}(x)\) is zero, \(f\) factors through \(\Fib(x\to L_{D(Z)}(x))=1_Z\otimes x\).
As \(z\otimes x=\mathrm{colim}\, z_i\otimes x\), the morphism \(f\) factors through some \(z_i\otimes x\in\ang{x}_{S}\cap D(Z)\), as desired.

For any \(a\in D\), there is a fiber sequence
\[
1_Z\otimes a\to a\to L_{S_Z}(1_S)\otimes a=L_{S_Z}(a),
\]
so in particular \(a\in D(Z)\) if and only if \(1_Z\otimes a\cong a\).
Moreover, if there exists \(a\in \Ind~ S_Z\) such that \(a\otimes b\cong b\) for all \(b\in S_Z\), then \(a\cong 1_Z\).

Let \(Z,W\subset \Spec(S,S)\) be closed subsets.
For any \(a\in S_{Z\cap W}\), we have \(1_Z\otimes 1_W\otimes a\cong a\), and \(1_Z\otimes 1_W\in \Ind ~S_Z\cap \Ind~ S_W=\Ind (S_Z\cap S_W)=\Ind ~ S_{Z\cap W}\).
By uniqueness, it follows that \(1_Z\otimes 1_W\cong 1_{Z\cap W}\).
This implies \(D(Z)\cap D(W)=D(Z\cap W)\).

Let \(\{Z_i\}\) be a family of closed subsets of \(\Spec(S,S)\).
To prove \(D(\bigcup Z_i)=\ang{\bigcup D(Z_i)}\), it suffices to treat the case \(S=D\).
In this case, the claim follows from \cite[Theorem~4.10]{Ba}.
\end{proof}

\section{Applications}\label{section:applications}
In this section, we illustrate how the relative Matsui spectrum appears in familiar geometric settings. 
We focus on derived categories of perfect complexes, twisted derived categories, bounded derived categories of coherent sheaves, categories of singularities of a scheme, and derived matrix factorization categories.  

Throughout this section, we fix \(B=\perf(\mathbb{Z})\).

\subsection{Some properties for calculating relative Matsui spectra}
In this section, we establish several basic properties that will be useful in explicit calculations of relative Matsui spectra. 
We first study the behavior of submodules under change of algebras, and then prove descent-type results.

Let \((D,S)\) and \((E,T)\) be data, and let \(f_2\colon S \to T\) be a homomorphism of algebras, \(f_1\colon D \to E\) a morphism of \(S\)-modules and \(f=(f_1,f_2)\).
In this setting, we define two maps between \(\Sub_S(D)\) and \(\Sub_T(E)\):
\begin{align*}
\ang{f}_T\colon \Sub_S(D) &\longrightarrow \Sub_T(E),\quad I \longmapsto \ang{f_1(I)}_T,\\
f^{-1} \colon \Sub_T(E) &\longrightarrow \Sub_S(D),\quad J \longmapsto f_1^{-1}(J).
\end{align*}

\begin{lem}\label{lem:change_of_algebras}
In the above setting, the following conditions are equivalent:
\begin{enumerate}
\item \(\ang{f}_T\) is injective.
\item \(f^{-1}\) is surjective.
\item For every prime \(S\)-submodule \(P \subset D\), we have \(\ang{f}_T(P) \neq \ang{f}_T(\ol{P})\).
\item For every prime \(S\)-submodule \(P \subset D\), there exists a prime \(T\)-submodule \(Q \subset E\) such that \(f^{-1}(Q) = P\) and \(f^{-1}(\ol{Q}) \supsetneq P\).
\item The composite \(f^{-1} \circ \ang{f}_T\) is the identity map.
\end{enumerate}
\end{lem}

\begin{proof}
(5) clearly implies (1), and (1) implies (3).

Assume (3).
Let \(P \subset D\) be a prime \(S\)-submodule and take \(p \in \ol{P} \setminus P\).
Then \(f(p) \notin \ang{f}_T(P)\) by (3).
By Lemma~\ref{lem:existance_of_prime}, there exists a prime \(T\)-submodule \(Q \subset E\) such that \(\ang{f}_T(P) \subset Q\) and \(p \notin Q\).
Hence (4) holds.

Assume (4).
Let \(I \subset D\) be any \(S\)-submodule.
Then there exists a family \(\{P_i\}\) of prime \(S\)-submodules such that \(I = \bigcap P_i\).
By (4), each \(P_i\) is the inverse image of some prime \(Q_i \subset E\).
Thus \[f^{-1}\!\left( \bigcap Q_i \right) = \bigcap f^{-1}(Q_i) = \bigcap P_i = I,\] and hence (2) holds.

Assume (2).
Given any \(I \subset D\), there exists \(J \subset E\) such that \(f^{-1}(J) = I\).
Then clearly, \(f(I) \subset \ang{f}_T(I) \subset J\), so \[I \subset f^{-1}(\ang{f}_T(I)) \subset f^{-1}(J) = I.\]
Hence, \(f^{-1} \circ \ang{f}_T = \id\), so (5) holds.
\end{proof}

Similarly, if \(E,T\) are presentable stable \(\infty\)-categories and \(E\) is a \(T\)-module, \(\Sub_T^{\mo{pr}}(E)\) denotes the set of presentable \(T\)-submodules.
For a subset \(W\subset E\), \(\ang{W}_T^{\mo{pr}}\) denotes the presentable \(T\)-submodule generated by \(W\).
Let \((D,S)\) be a datum, \(f_2:\Ind ~S\to T\) a homomorphism of algebras, and \(f_1:\Ind~ D\to E\) a morphism of \(\Ind~ S\)-modules. 
In this setting, we define
\[
\ang{f}_T^{\mo{pr}}:\Sub_S(D)\longrightarrow \Sub_T^{\mo{pr}}(E),\quad 
I \longmapsto \ang{f_1(I)}_T^{\mo{pr}}.
\]
Note that \(\Sub_S(D)\) naturally embeds into \(\Sub_{\Ind S}(\Ind D)\) via \(I \mapsto \Ind I\).

\begin{lem}[Descent]\label{lem:descent}
  Let \((D,S)\) be a datum, and let \(S'\) be an algebra object with homomorphism \(\phi:S\to S'\).
  Assume \(\Ind(\phi)\) has an adjoint \(\psi\), that \(\psi(s')\otimes s\cong \psi(s'\otimes \phi(s))\) for all \(s\in S, s'\in S'\), and that
  \(\ang{\psi(1_{S'})\otimes d}^{\mo{pr}}_{\Ind S}=\ang{d}^{\mo{pr}}_{\Ind S}\subset \Ind(S\otimes D)\) for all \(d\in D\) (e.g., \(\ang{\psi(1_{S'})}^{\mo{pr}}_{\Ind S}=\Ind S\)).
  Then \((\phi\otimes \Id_D,\phi):(D,S)\to (S'\otimes_S D,S')\) satisfies the equivalent conditions of Lemma~\ref{lem:change_of_algebras}.

  Moreover, if the restriction \(\psi':=\psi|_{S'}\) factors through \(S\), and \(\ang{\psi'}_S:\Sub_{S'}(S'\otimes_S D)\to\Sub_S(D)\) is injective, then
  \(\ang{\phi\otimes \Id_D}_{S'}:\Sub_S(D)\to \Sub_{S'}(S'\otimes_S D)\) is bijective.
\end{lem}
\begin{proof}
  The assumptions imply that the composition 
  \begin{align*}
  \ang{\phi\otimes \Id_D}_{S'}&:\Sub_S(D)\to \Sub_{S'}(S'\otimes_S D),\quad I\mapsto \ang{\phi\otimes \Id_D(I)}_{S'}, \quad\text{with}\\
  \ang{\psi|_{S'}\otimes \Id_{D}}^{\mo{pr}}_{\Ind S}&:\Sub_{S'}(S'\otimes_S D)\to \Sub_{\Ind S}^{\mo{pr}}(\Ind D),\quad J\mapsto \ang{\psi\otimes \Id_{D}(J)}^{\mo{pr}}_{\Ind S},\end{align*}
  is the map \(I\mapsto \Ind I\).
  Since this is injective, so is \(\ang{\phi\otimes \Id_D}_{S'}\).

  In the case of the last assertion, \(\ang{\psi|_{S'}\otimes \Id_{D}}^{\mo{pr}}_{\Ind S}\) factors through \(\Sub_S(D)\); it is injective, and its composite with \(\ang{\phi\otimes \Id_D}_{S'}\) is the identity on \(\Sub_S(D)\).
  This implies the bijectivity of \(\ang{\phi\otimes \Id_D}_{S'}\).
\end{proof}

\begin{lem}\label{lem:functor_of_thick_subcategories}
Let \(S,S',T\) be commutative algebra objects in \(\Perf\) with homomorphisms \(S\to S'\) and \(S\to T\).
Let \(D\) be a left \(S'\)-module and \(D'\) a left \(T\)-module.
Let \(\Phi,\Psi:D\to D'\) be \(S\)-linear functors, and set \(T'=S'\otimes_S T\).
Assume \(\ang{\Ind \Phi}_{\Ind T'}^{\mo{pr}}=\ang{\Ind \Psi}_{\Ind T'}^{\mo{pr}}\) in \(\mo{Func}_{\Ind S}^{\ld}(\Ind D,\Ind D')\).
Then the induced maps \(\ang{\Phi_F}_{S'},\ang{\Phi_G}_{S'}:\Sub_S(D)\longrightarrow \Sub_T(D')\) are the same.
\end{lem}

\noindent
Note that \(\mo{Func}_{\Ind S}^{\ld}(\Ind D,\Ind D')\), the \(\infty\)-category of \(\Ind~ S\)-linear colimit-preserving functors, is a presentable stable \(\infty\)-category, and it carries a natural action of \(\Ind T'\).

\begin{proof}
Let \(F'\to F\to F''\) be an exact triangle in \(\mo{Func}_{\Ind S}^{\ld}(\Ind D,\Ind D')\).
Then, for any \(d\in D\), the sequence \(F'(d)\to F(d)\to F''(d)\) is an exact triangle.
Similarly, the construction preserves shifts and direct sums.
For any \(s\in S',\ t\in T,\ F\in \mo{Func}_{\Ind S}^{\ld}(\Ind D,\Ind D')\), and \(I\in \Sub_S(D)\), \[\ang{((s\otimes t)\otimes F)(I)}_T=\ang{t\otimes (F(s\otimes I))}_{T}\subset \ang{F(I)}_T.\]
This implies the desired equality. 
\end{proof}

\subsection{Derived categories of perfect complexes of schemes}
We recall that for any scheme \(X\), the triangulated category \(\perf(X)\) is classically finitely generated by \cite[Corollary 3.1.2]{BB}.
For a subset \(U \subset X\) of the form \(U = W \setminus Z\) with specialization-closed subsets \(W\) and \(Z\), we denote by \(\perf_U(X)\), \(\coh_U(X)\), etc., the corresponding categories as in Theorem \ref{thm:local_decomposition}.

\begin{lem}\label{lem:diagonal_of_Koszul_algebra}
Let \(R\) be a commutative ring, \(r_1,\ldots,r_n\in R\), and let 
\(A=\mo{Kos}(R;r_1,\ldots,r_n)\) be the Koszul algebra.
Then,
\[
\ang{A}_{A\dotimes_R A}^{\mo{pr}}
= \ang{A\dotimes_R A}_{A\dotimes_R A}^{\mo{pr}}
\quad \text{in } \qcoh(A\dotimes_R A).
\]
\end{lem}
\begin{proof}
Since \(\mo{Kos}(R; r_1,\ldots,r_n) \cong \bigotimes_{i=1}^n \mo{Kos}(R; r_i)\), it suffices to treat the case \(n=1\).
Put \(B=A\dotimes_R A\) and \(r=r_1\).  
Now, \(A\) is a dg-algebra of the form \(R\os{r}{\to} R\), and 
\(B\) is of the form \(R\to R\oplus R\to R\).  
The kernel of the diagonal map \(B\to A\) is \(A[1]\).  
Thus, there is a free resolution \(F_*=(\cdots \to B[2]\to B[1]\to B)\) of \(A\) in the category of chain complexes of dg-\(B\)-modules.  
The stupid truncation \(\tau_{\ge 1}F_*\) is a shift of a free resolution of \(A[1]\), and the cocone of the canonical map \(F_*\to \tau_{\ge 1}F_*\) is \(B\).  
This proves the assertion.
\end{proof}

\begin{lem}\label{lem:quasi_affine}
Let \(f:X\to Y\) be a quasi-affine morphism of schemes.  
Then, \(\perf(X)\) is classically generated by \(\ld f^*\perf(Y)\).
\end{lem}
\begin{proof}
It suffices to prove that \(\qcoh(X)\cong \Ind~\perf(X)\) is generated by \(\ld f^*\qcoh(Y)\) as a presentable thick subcategory.  
Let \(\{U_i\}_{i=1}^n\subset Y\) be an affine open covering, and set \(V_i=f^{-1}(U_i)\).  
For any \(F\in\qcoh(X)\), there exists an exact triangle 
\(F'\to F\to F''\) with 
\(F'\in \qcoh_{X\setminus V_1}(X)\subset \qcoh(\bigcup_{i=2}^n V_i)\) 
and \(F''\in \qcoh(V_1)\).  
Thus, by induction, we may assume \(n=1\).  
In this case, \(X\) is quasi-affine, and \(\qcoh(X)\) is generated by 
\(\sO_X=\ld f^*\sO_Y\).
\end{proof}

\begin{prop}\label{prop:schemes_are_rigid_datum}
Let \(f:X\to Y\) be a morphism of schemes.  
Then, \((\perf(X),\perf(Y))\) is a rigid datum.
\end{prop}
\begin{proof}
This follows from \cite[Theorem 6.3]{Ba} and the facts that 
\(\perf(X)\) is a rigid tensor triangulated category 
and \(\ld f^*\) preserves dual objects.
\end{proof}

\begin{thm}\label{thm:example_of_perfect_complexes}
Let \(f:X\to Y\) be a morphism of schemes, \((J,\perf(Y))\) a rigid datum, 
and \(\pi_J:\Spec(J,\perf(Y))\to Y\) the morphism of 
Theorem \ref{thm:morphism_to_the_base} with the canonical isomorphism 
\(\oSpec\Ho(\perf(Y))\cong Y\) (\cite{Ba}).  
Assume there is a \(\perf(Y)\)-linear equivalence 
\(\Phi:\perf(X)\to \perf(Y)\).  
Then the following statements hold.
\begin{enumerate}
  \item There is a canonical open embedding 
  \[\iota_{\Phi}:X\to \Spec(J,\perf(Y))\] 
  such that \(f\) decomposes as \(\pi_J\circ\iota_\Phi\).
  \item If \(X=Y\) and \(f=\Id_X\), then \(\Spec(J,\perf(X))\cong X\).
  \item Let \(g:Y\to Z\) be a morphism of schemes.  
  Then there is a canonical open embedding \(\Spec(J,\perf(Y))\to \Spec(J,\perf(Z))\).  
  If \(g\) is quasi-affine, then it is an isomorphism.
  \item If \(Y\) is quasi-affine, then \(\Spec(J,\perf(Y))\cong \tSpec(J)\).
  \item Let \(h:Y'\to Y\) be a morphism of schemes, \(W\subset Y\) the image of \(h\), \(X'=X\dotimes_Y Y'\), and \(\mo{pr}:X'\to X\) the projection.  
  Then, for any prime \(\perf(Y)\)-submodule \(P\subset \perf(X)\) in the set-theoretic inverse image \(\pi^{-1}_{\perf(X)}(W)\subset \Spec(\perf(X),\perf(Y))\), there exists a prime \(\perf(Y')\)-submodule \(Q\subset \perf(X')\) such that \((\ld \mo{pr}^*)^{-1}(Q)=P \).
  \item Let \(Y'\) be a simplicial \(Y\)-scheme locally isomorphic to the spectrum of a Koszul algebra of a localization of \(\sO_Y\).  
  Let \(X',\mo{pr},W\) be as above.  
  Then there is a canonical homeomorphism
  \[\Spc(\perf(X),\perf(Y))_{W}\cong \Spc(\perf(X'),\perf(Y')).\]  
  In particular, if \(y\in Y\) is a regular point such that \(f\) is flat over \(y\), then there is a homeomorphism \(\Spc(\perf(X),\perf(Y))_y\cong \tSpc(\perf(X_y)).\)
\end{enumerate}
\end{thm}
\begin{proof}
First, let us prove (2).  
We have a homeomorphism \(\Spc(\perf(X),\perf(X))\cong \oSpc(\perf(X))\) by Theorem \ref{thm:comparison_with_the_Balmer_spectrum} and an isomorphism \(\oSpec(\perf(X))\cong X\) (\cite[Theorem 6.3]{Ba}).  
By \cite{BFN}, for an affine open subset \(U\subset X\), there is an isomorphism
\begin{align*}
\Nat^0_{\perf(\intg)}(\perf(U),\perf(U))
&= \Nat^0_{\qcoh(\intg)}(\qcoh(U),\qcoh(U))\\
&=\Hom^0_{\qcoh(U\dotimes U)}(\sO_{\Delta},\sO_{\Delta})
= \Gamma(U,\sO_X).
\end{align*}
These isomorphisms give an isomorphism of structure sheaves.

(3) follows from Theorem \ref{thm:restriction_of_the_base} and Lemma \ref{lem:quasi_affine}.

For (1), the existence of the canonical open embedding follows from (2) and (3).  
The calculation of the composition \(\pi_J\circ \iota_\Phi\) is straightforward.

(4) follows from (3) in the quasi-affine case and from 
Theorem \ref{thm:comparison_with_the_Matsui_spectrum}.

For (5), after localizing, we may assume \(Y\) has a unique closed point \(s\), 
\(P\in\pi_{\perf(X)}^{-1}(s)\), and \(h\) is affine.  
By Theorem \ref{thm:local_decomposition}, we may replace \(\perf(X)\) by \(\perf_{\pi_{f^{-1}(s)}}(X)\).  
Then, by Lemma \ref{lem:descent}, it suffices to show that 
\(\ang{\rd f_* (\sO_{Y'})}_{\perf(Y)}^{\mo{pr}}\) contains \(\qcoh_{\{s\}}(Y)\).  
This follows from \cite{Ne}.

For (6), after localizing, we may assume \(Y\) is affine and \(Y'\) is a spectrum of a Koszul \(\sO_Y\)-algebra.  
By Lemma \ref{lem:descent} and the above discussion, it suffices to show that 
\(\ang{\ld \mo{pr}^*\rd \mo{pr}_*}_{\perf(X')}:
\Sub_{\perf(Y')}\perf(X')\to \Sub_{\perf(Y')}\perf(X')\) is bijective.  
The functors \(\ld \mo{pr}^*\rd \mo{pr}_*\) and \(\Id_{\perf(X')}\) are represented by
\[\sO_{\Delta_X}\boxtimes_Y (\sO_{Y'}\dotimes_{\sO_Y}\sO_{Y'}), 
\sO_{\Delta_X}\boxtimes_Y\sO_{\Delta_{Y'}} \in 
\qcoh(X\dotimes_{Y} X\dotimes_Y Y'\dotimes_Y Y')\cong \qcoh(X'\dotimes_{Y} X')\] as integral transforms, respectively.  
By Lemma \ref{lem:functor_of_thick_subcategories}, it is enough to prove 
\(\ang{\sO_{\Delta_{Y'}}}^{\mo{pr}}=\qcoh(Y'\dotimes_Y Y')\), 
which follows from Lemma \ref{lem:diagonal_of_Koszul_algebra}.
\end{proof}

\begin{rem}
It is known that the canonical embedding \(X \to \tSpec\perf(X)\) is an open embedding when \(X\) is quasi-affine~\cite[Corollary~4.7]{Ma2} or quasi-projective over an algebraically closed field~\cite[Theorem~3.2]{IM}.  
In fact, by Theorem~\ref{thm:example_of_perfect_complexes}, it is an open embedding whenever \(X\) is a (topologically) Noetherian scheme.  
Note that if one uses the original definition of Matsui spectra, \(X\) must be replaced by its reduced scheme \(X_{\mo{red}}\); see~\cite{Ma2}.
\end{rem}

\begin{cor}
  Let \(X\) be a scheme, and let \(f:Y\to X\) be a simplicial \(X\)-scheme locally isomorphic to the spectrum of a Koszul \(\sO_X\)-algebra \(\mo{Kos}(\sO_X;f_1,\ldots,f_n)\).
  Let \(I\) be the ideal sheaf locally generated by \(f_1,\ldots,f_n\), and let \(X_0\) be the closed subscheme \(\Spec(\sO_X/I)\subset X\).
  Then \(\Sub_{\perf(Y)}(\perf(Y))=\Sub_{\perf(X)}(\perf(Y))\), and this set is in bijection with specialization-closed subsets of \(X_0\).
  Moreover, \(X_0\cong \oSpec(Y)\cong \Spec(Y,\perf(X))\).
\end{cor}
\begin{proof}
By the proof of Theorem \ref{thm:example_of_perfect_complexes}, \(\Sub_{\perf(Y)}(\perf(Y))\) corresponds bijectively to \(\Sub_{\perf(X)}(\perf_{X_0}(X))\), and by \cite[Theorem 3.15]{Th}, \(\Sub_{\perf(X)}(\perf_{X_0}(X))\) corresponds bijectively to specialization-closed subsets of \(X_0\).
Since \(\perf(Y)\) is classically generated by \(\ld f^*\perf(X)\), we have \(\Sub_{\perf(Y)}(\perf(Y))=\Sub_{\perf(X)}(\perf(Y))\).
This classification yields the homeomorphisms \(X_0\cong \oSpc(Y)\cong \Spc(Y,\perf(X))\).
To prove the isomorphisms of structure sheaves, after localization we may assume that \(X\) is affine and \(Y\) is the spectrum of a Koszul \(\sO_X\)-algebra \(\mo{Kos}(\sO_X;f_1,\ldots,f_n)\), written as \(R\).
The natural isomorphism \(\End^0_{\perf(R)}(R)\cong \pi_0(R)\) gives the first isomorphism, and the second follows as in Theorem \ref{thm:example_of_perfect_complexes} (2).
\end{proof}

\begin{defn}
Let \((D,S)\) be a datum.
We say that a subset \(X\subset \Spec(D,S)\) is \emph{commutative} if, for any \(U,V\subset X\), the corresponding thick subcategories \(I_U,I_V\) commute.
\end{defn}

\begin{prop}\label{prop:maximal_commutative_subset}
Let \(X\to Y\) be a morphism of schemes, and let \(\iota:X\to \Spec(\perf(X),\perf(Y))\) be the natural open embedding.
Then the image of \(\iota\) is a maximal commutative subset.
\end{prop}

\begin{proof}
Since \((\perf(X),\perf(X))\) is a rigid datum, the image of \(\iota\) is commutative.  
We now prove maximality.  
Let \(I\subset \perf(X)\) be a thick subcategory commuting with \(I_U\) for any \(U\subset X\).  
Take a finite affine open covering \(\{U_i\}_i\) of \(X\), and set \(Z=\Supp(I)\cap X\) and \(Z_i=X\setminus U_i\).  
We then have \(\ang{\perf_{Z_i}(X)\cup I}=\perf_{Z_i\cup Z}(X)\),  
since submodules containing \(\perf_{Z_i}(X)\) correspond bijectively to submodules of \(\Idem(\perf(X)/\perf_{Z_i}(X))\cong \perf(U_i)\), by the classification~\cite[Theorem~3.15]{Th} together with the fact that, for an affine scheme, every thick subcategory is a tensor ideal.  
By Lemma~\ref{lem:submodule_reconstruction}, it follows that \(I=\perf_Z(X)\), which proves maximality.  
\end{proof}

\begin{prop}
Let \(Y\) be a scheme, and let \(X,X'\) be \(Y\)-schemes.
Suppose \(\Phi:\perf(X)\to \perf(X')\) is an equivalence of \(\perf(Y)\)-modules.
Define
\[
X(\Phi)=\{x\in X \mid \exists x'\in X' \ \text{such that}\ \Phi(P_x)=P_{x'}\},
\]
where \(P_x,P_{x'}\) are corresponding prime submodules.
Then \(X(\Phi)\) is an open subscheme of \(X\), and there is an isomorphism \(X(\Phi)\cong X'(\Phi^{-1})\).
Moreover, \(X(\Phi)=X\) if and only if \(X'(\Phi^{-1})=X'\).
\end{prop}

\begin{proof}
Let \(\iota_1,\iota_2:X,X'\to \Spec(\perf(X),\perf(Y))\) be the open embeddings associated with \(\Id_{\perf(X)}\) and \(\Phi^{-1}\), respectively.
Then we have \(X(\Phi)\cong \iota_1(X)\cap \iota_2(X')\cong X'(\Phi^{-1})\).
The last assertion follows from Proposition \ref{prop:maximal_commutative_subset}.
\end{proof}

\subsection{Twisted derived categories and Severi--Brauer schemes}

In this subsection, we examine twisted derived categories and Severi--Brauer schemes. 
We show that relative spectrum of twisted perfect complexes recovers the base scheme. 
For Severi--Brauer schemes, the spectrum contains the scheme itself together with countably many copies of the base. 
In relative dimension one, we obtain an isomorphism and classify submodules via specialization-closed subsets.

Let \(X\) be a scheme, and let  
\(\mu = (\mu_0, \mu_1) \in \etch^2(X, \mathbf{G}_m) \times \etch^1(X, \mathbb{Z})\).  
The corresponding twisted derived category \(\perf(X, \mu)\) is naturally a \(\perf(X,0)=\perf(X)\)-module; see \cite{To2} for the definition.  
There is a canonical isomorphism
\[
\perf(X, \mu) \otimes_{\perf(X)} \perf(X, \mu') \;\cong\; \perf(X, \mu + \mu').
\]
When \(\mu_1=0\), the category \(\perf(X, \mu)\) is called the derived category of perfect complexes of twisted sheaves; see \cite{Ca} for details.

\begin{prop}
Let \(X\) be a scheme, \(\mu\in \etch^2(X,\mathbf{G}_m)\times \etch^1(X,\mathbf{Z})\).
Then \((\perf(X,\mu),\perf(X))\) is a rigid datum.
\end{prop}

\begin{proof}
The argument is analogous to Proposition \ref{prop:schemes_are_rigid_datum}.
\end{proof}

\begin{thm}
Let \(X\) be a scheme and \(\mu\in \etch^2(X,\mathbf{G}_m)\times \etch^1(X,\mathbf{Z})\).  
Consider the morphism
\[
\pi:\Spec(\perf(X,\mu),\perf(X))\to X
\]
of Theorem \ref{thm:morphism_to_the_base}, equipped with the canonical isomorphism 
\(\oSpec\Ho(\perf(X))\cong X\).
Then \(\pi\) is an isomorphism.
\end{thm}

\begin{proof}
Let \(F\in \perf(X,\mu)\) and \(F'\in \perf(X,-\mu)\) be classical generators, whose existence is shown in \cite{To2}.
Then \(F\otimes F'\) and \(F'\otimes F\) are classical generators of \(\perf(X)\).
Hence the maps
\[
\Sub_{\perf(X)}(\perf(X,\mu))\to \Sub_{\perf(X)}(\perf(X)),\quad
J\mapsto \ang{F'\otimes J}_{\perf(X)},
\]
and
\[
\Sub_{\perf(X)}(\perf(X))\to \Sub_{\perf(X)}(\perf(X,\mu)),\quad
I\mapsto \ang{F\otimes I}_{\perf(X)},
\]
are inverse to each other, and thus bijections.
Therefore, \(\pi\) is a homeomorphism.

Next, we show the isomorphism of structure sheaves.  
After localization, we may assume that \(X\) is an affine scheme \(\Spec R\).  
Put \(A=\sNat_{\perf(\mathbb{Z})}(\Id_{\perf(X,\mu)},\perf(X,\mu))\), where \(\sNat\) denotes the spectrum of natural transformations.  
For affine étale morphisms \(U\cong \Spec R' \to X\), the assignments
\[
U\longmapsto \sNat_{\perf(\mathbb{Z})}(\Id_{\perf(U,\mu|_U)},\perf(U,\mu|_U)), 
\quad
U\longmapsto \rd \Gamma(U,\sO_U)
\]
are given by \((R'\dotimes_{\mathbb{Z}} R')\dotimes_{R\dotimes_{\mathbb{Z}}R}A\) and \(R' \cong (R'\dotimes_{\mathbb{Z}}R')\dotimes_{R\dotimes_{\mathbb{Z}}R}R\), respectively.  
Since \(\mu\) is étale locally trivial, their \(\pi_0\) are locally isomorphic, as shown in Theorem~\ref{thm:example_of_perfect_complexes}.  
The homomorphism \(R\otimes_{\mathbb{Z}}R \to R'\otimes_{\mathbb{Z}}R'\) is flat in the sense of \cite{Lu2}, hence the collections of these \(\pi_0\)'s give rise to usual quasi-coherent sheaves.
By descent, they are isomorphic.
\end{proof}

\begin{prop}\label{prop:Severi-Brauer_scheme}
Let \(X\) be a scheme, and let \(Y\to X\) be a Severi--Brauer scheme of relative dimension \(r>0\).
Then there is a natural open embedding
\[
Y\;\amalg\; \left(\amalg_\mathbb{Z} X\right)\;\longrightarrow\; \Spec(\perf(Y),\perf(X)).
\]
\end{prop}

\begin{proof}
The embedding from \(Y\) follows from Theorem~\ref{thm:example_of_perfect_complexes}.
Let \(\mu \in \etch^2(X, \mathbf{G}_m)\) be the class corresponding to \(Y\).
By \cite{Be}, for any integer \(k\), the category \(\perf(Y)\) admits a semi-orthogonal decomposition of the form
\((\perf(X, Y)_k, \ldots, \perf(X, Y)_{k+r})\).
Each component \(\perf(X, Y)_k\) is equivalent to \(\perf(X, k\mu)\), and each is classically finitely generated by \cite{To2}.
This yields an open embedding
\begin{align*}
\iota_k : X &\cong \Spec(\perf(X, k\mu), \perf(X))\cong \Spec(\perf(X,Y)_k,\perf(X)) \\
&\cong \Spec\!\left(\Idem\!\left(\perf(Y) / \ang{\perf(X, Y)_{k+1}, \ldots, \perf(X, Y)_{k+r}}\right), \perf(X)\right)\\
&\to \Spec(\perf(Y), \perf(X)).
\end{align*}
For each \(k\), the image of \(\iota_k\) is disjoint from \(Y\), since \(Y\subset \Supp(\perf(X,Y)_{k+1})\) while \(\iota_k(X)\cap \Supp(\perf(X,Y)_{k+1})=\emptyset\). 
Similarly, the images of \(\iota_k\) and \(\iota_{k+1}\) are disjoint.
\end{proof}

\begin{thm}
Let \(X\) be a scheme, and let \(Y\to X\) be a Severi--Brauer scheme of relative dimension \(1\).
Then there is an isomorphism
\[
Y\;\amalg\; \left(\amalg_\mathbb{Z} X\right)\;\longrightarrow\; \Spec(\perf(Y),\perf(X)).
\]
Moreover, the set of submodules \(\Sub_{\perf(X)}(\perf(Y))\) is in bijection with subsets \(Z\subset \Spec(\perf(Y),\perf(X))\) satisfying:
\begin{enumerate}
\item \(Z\) is specialization-closed.
\item For each \(x\in X\), the fiber \(Z_x\) is one of the following:
\begin{enumerate}
\item the empty set,
\item there exists a non-empty specialization-closed subset \(W\subset Y_x\) such that \(Z_x=W\;\amalg\;\bigl(\amalg_{\mathbb{Z}}\{x\}\bigr)\), or
\item there exists some \(n\in\mathbb{Z}\) such that \(Z_x=Y_x\;\amalg\;\bigl(\amalg_{\mathbb{Z}\setminus\{n\}}\{x\}\bigr)\).
\end{enumerate}
\end{enumerate}
\end{thm}

\begin{proof}
If \(X\) is the spectrum of an algebraically closed field, then \(Y\) is a projective line, and the claim follows from \cite{KS}.
In general, the morphism is given by Proposition \ref{prop:Severi-Brauer_scheme}, and it is an isomorphism by Theorem \ref{thm:example_of_perfect_complexes} (5), taking \(Y'\) to be the spectrum of an algebraically closed field.

Next, we classify submodules of \(\perf(Y)\).
The correspondence is given by the support map, which is always injective by Theorem \ref{thm:classification_of_submodules}.
Condition (1) is standard.
For (2), let \(s = \Spec~ \overline{k}\) be the spectrum of an algebraically closed field, and take a morphism \(s \to X\).
Define
\[
\phi:\tSpec(\perf(Y_{s}))\;\cong\; Y_{s}\;\amalg\;\left(\amalg_\mathbb{Z} \Spec\ s\right)\;\longrightarrow\; Y\;\amalg\;\left(\amalg_{\mathbb{Z}}X\right)\;\cong\;\Spec(\perf(Y),\perf(X)).
\]
Then the following diagram commutes:
\[
\xymatrix{
\Sub(\perf(Y_{s}))\ar[d]_{\Supp}&\Sub_{\perf(X)}(\perf(Y))\ar[l]_{\text{derived pullback}}\ar[d]^{\Supp}\\
\{Z'\subset \Spec(\perf(Y_{s}))\}& \ar[l]^{\phi^{-1}}\{Z\subset \Spec(\perf(Y),\perf(X))\}.}
\]
Thus, the necessity of (2) follows.

Now, to verify existence: given a subset \(Z\) satisfying (1) and (2), set
\begin{align*}
U&=\{x\in X\mid Z_x \text{ is of type (b)}\},\quad
U'=Z\cap Y_U,\\
V_m&=\{x\in X\mid Z_x \text{ is of type (c) with }n=m,\ \text{or of type (b) with }W=Y_x\}.
\end{align*}
Here \(U,V_m\) are specialization-closed subsets of \(X\), and \(U'\) is specialization-closed in \(Y\).
Define \(I=\ang{\perf_{U'}(Y)\cup \bigcup_{m\in\mathbb{Z}}\ \bigcup_{a\in\perf_{V_m}(X)} a\otimes\perf(X,Y)_{m+1}}\).
It is straightforward to check that \(\Supp(I)=Z\).
\end{proof}

\subsection{Koszul algebras}\label{subsec:Koszul_algebras}

Liu and Pollitz~\cite{liu2025classifying} gave a beautiful classification of thick subcategories over Koszul algebras of regular rings.  
Our construction generalizes their results to the setting of regular schemes.

For a graded commutative ring \(A\), let \(\Spec^{\mo{gr}} A\) denote the set of homogeneous prime ideals of \(A\) with the usual topology.  
There is a bijection \(\Spec^{\mo{gr}} A \cong \mo{Proj}(A) \coprod \Spec(A_0)\).  
Let \(R\) be a commutative ring, and let \(f_1,\ldots,f_n \in R\).  
Set \(E = \mo{Kos}(R; f_1,\ldots,f_n)\), the Koszul \(R\)-algebra.  
Let \(A = R[x_1,\ldots,x_n]\), where each \(x_i\) is assigned homological degree \(-2\).
Define \(\Supp^{\mo{coh}}(\Spec E)\) and \(\Supp^{\mo{sing}}(\Spec E)\) to be the subspaces of \(\Spec^{\mo{gr}}(A)\), denoted \(\Supp(A,w)\) and \(\Sing(\mo{Proj}(A/(w)))\) in~\cite{liu2025classifying}, respectively.  
There is a bijection
\[
\Supp^{\mo{coh}}(\Spec E) = \Supp^{\mo{sing}}(\Spec E) \coprod \Spec R/(f_1,\ldots,f_n).
\]
For a point \(x\in \Spec E\) (identified with \(\Spec R/(f_1,\ldots,f_n)\) as a topological space), the topological fiber \(\Supp(\Spec E)_x\) is given by
\[
\Supp(\Spec E)_x \cong \Spec~ \mo{Sym}_{k(x)}(k(x)^{c(x)}[-2]),
\]
where \(c(x) = n - \bigl(\dim(\sO_{\Spec R,x}) - \mo{embdim}(\sO_{R/(f_1,\ldots,f_n),x})\bigr)\).

Let \(Y\) be a simplicial scheme locally isomorphic to the spectrum of a Koszul algebra \(\mo{Kos}(R;f_1,\ldots,f_n)\) of a regular ring \(R\).  
Denote by \(\mo{D}^f(Y)\) the full subcategory of \(\qcoh(Y)\) consisting of bounded complexes with \(\pi_0(\sO_Y)\)-coherent cohomology.  
By gluing together the spaces \(\Supp^{\mo{coh}}(\Spec E)\) (resp.~\(\Supp^{\mo{sing}}(\Spec E)\)), one obtains the topological space \(\Supp^{\mo{coh}}(Y)\) (resp.~\(\Supp^{\mo{sing}}(Y)\)) over \(Y\).

\begin{prop}
The action of \(\perf(Y)\) on \(\mo{D}^f(Y)\) is rigid.
\end{prop}

\begin{thm}
In the above setting, there is a bijection
\[
\Spc(\mo{D}^f(Y),\perf(Y)) \cong \Supp^{\mo{coh}}(Y).
\]
The family of closed subsets of \(\Spc(\mo{D}^f(Y),\perf(Y))\) is generated by those subsets \(Z \subset \Spc(\mo{D}^f(Y),\perf(Y))\) that are locally closed subspaces of \(\Spec(A)\).  
Submodules of \(\mo{D}^f(Y/X)\) correspond bijectively to specialization-closed subsets of \(\Supp^{\mo{coh}}(Y)\).
\end{thm}

\begin{proof}
If \(X\) is affine, this follows from~\cite{liu2025classifying}.  
For general \(X\), one patches the local results.
\end{proof}

If \(f_1,\ldots,f_n\) form a regular sequence, then \(Y\) is a classical scheme which is locally a complete intersection.  
In this case, \(\mo{D}^f(Y)=\coh(Y)\), and explicit arguments are given in the next subsection.

\subsection{The categories related to singularities}

In this subsection, we investigate the relative spectra associated with categories of singularities and coherent sheaves. 
We establish comparison diagrams and characterize fibers in terms of regularity, complete intersections, and hypersurface singularities. 
This provides a geometric interpretation of singular loci via the spectra of these categories.

A topological space \(S\) is called \emph{local} if \(S\) has a unique closed point \(s\), and every point \(s'\in S\) is a generalization of \(s\).

The embedding codimension \(\mo{ecodim}(R)\) of a local ring \(R\) is defined as the difference between its embedding dimension and its Krull dimension.  
A local complete intersection ring is called a \emph{hypersurface} if its embedding codimension is \(0\) or \(1\).

\begin{thm}\label{2-19}
Let \(X\to Y\) and \(Y\to Y'\) be morphisms of schemes.
Then there exists the following commutative diagram:
\[
\xymatrix{
\Spec^{\mo{fin}}(\Sing(X),\perf(Y))\ar[r]^s\ar[d]^{b'}&
\Spec^{\mo{fin}}(\coh(X),\perf(Y))\ar[d]^b\ar[r]&
Y\ar[d]\\
\Spec^{\mo{fin}}(\Sing(X),\perf(Y'))\ar[r]^{s'}&
\Spec^{\mo{fin}}(\coh(X),\perf(Y'))\ar[r]&
Y'}.
\]
If \(\coh(X)\) is finitely generated (e.g., when \(X\) is a separated scheme of finite type over a perfect field \cite[Theorem 7.39]{Ro}), then \(b,b'\) are open embeddings; otherwise they are injections.
The maps \(s,s'\) are always open embeddings.

Put \(X^{\mo{coh}}=\Spc(\coh(X),\perf(X))\) and \(X^{\mo{sing}}=\Spc(\Sing(X),\perf(X))\).
Fix a point \(x\in X\).
Then the following statements hold:
\begin{enumerate}
\item The following are equivalent: \(X\) is regular at \(x\); the fiber \((X^{\mo{coh}})_x\) consists of a single point; and the fiber \((X^{\mo{sing}})_x\) is empty.
\item \(X\) is a complete intersection at \(x\) if and only if the space \((X^{\mo{coh}})_x\) is local.
\item If \(X\) is locally a complete intersection, then there are bijections
\begin{align*}
X^{\mo{coh}} &\cong \Supp^{\mo{coh}}(X)\cong \Supp^{\mo{sing}}(X)\coprod X,\\
X^{\mo{sing}} &\cong \Supp^{\mo{sing}}(X),
\end{align*}
where \(\Supp^{\mo{coh}}(X)\) and \(\Supp^{\mo{sing}}(X)\) are the spaces defined in the previous subsection.  
In particular, there is a canonical closed embedding \(X \hookrightarrow X^{\mo{coh}}\).
Submodules of \(\coh(X)\) and \(\Sing(X)\) correspond bijectively to specialization-closed subsets of \(\Supp^{\mo{coh}}(X)\) and \(\Supp^{\mo{sing}}(X)\), respectively.  
For a point \(x\in X\), there are bijections
\[
(X^{\mo{coh}})_x \cong \mathbb{P}_{k(x)}^{\mo{ecodim}(\sO_{X,x})-1}\coprod \{x\},\quad 
(X^{\mo{sing}})_x \cong \mathbb{P}_{k(x)}^{\mo{ecodim}(\sO_{X,x})-1}.
\]

\item If \(X\) is a hypersurface at every point, then there is a homeomorphism
\[
X^{\mo{sing}}\cong \{\text{singular locus of }X\}.
\]
\end{enumerate}
\end{thm}

\begin{proof}
The maps \(s,s'\) of underlying sets are constructed as in Proposition \ref{prop:open_embedding}, and \(b,b'\) as in Theorem \ref{thm:restriction_of_the_base}.  
Commutativity follows from the naturality of these constructions.  
Continuity of \(s,s',b,b'\) also follows, since the topology on \(\Spec^{\mo{fin}}\) is the weakest one containing that of \(\Spec\) and making the maps to the spectrum of the base continuous.

(1) If \(X\) is regular at \(x\), the claim is clear.
Suppose \(X\) is not regular at \(x\).
Then \(\coh_x(X)\) has at least three distinct submodules: \(0\), \(\perf_x(X)\), and \(\coh_x(X)\).
By Theorem \ref{thm:classification_of_submodules}, \(\Spc(\coh_x(X))\) must have at least \(\lceil\log_2 3\rceil=2\) points.
Since \(\Sing_x(X)\) is nontrivial, Corollary \ref{cor:empty_spectrum} implies \(X^{\mo{sing}}\neq\emptyset\).

(2) This follows directly from \cite{Po}; see also Remark \ref{rem:T_0_ness}.

(3) follows from the previous subsection.  
Note that the topology on \(X^{\mo{coh}}\) is generated by the topology on \(\Spc(\coh(X),\perf(X))\) (as explained in detail in the previous subsection) together with the inverse images of open subsets of \(X\).

(4) After localization, assume \(X\) is affine.
Then the statement follows from \cite[Corollary 3.8]{Ma1}.
\end{proof}

\subsection{Derived matrix factorization categories}

Let \(X\) be a scheme, \(L\) a line bundle on \(X\), and \(W \in \Gamma(X, L)\).
The derived matrix factorization category \(\mathrm{DMF}(X, L, W)\) was introduced by Positselski~\cite{Pos}, \cite{EP}.  
There is a natural action of \(\mathrm{DMF}(X, L, 0)\) on \(\mathrm{DMF}(X, L, W)\).  

Let \(\iota:Z \to X\) be a closed subscheme.  
Hirano~\cite{Hi} introduced the relative singular locus \(\mathrm{Sing}(Z/X)\) by
\[
\mathrm{Sing}(Z/X)
= \{\, z\in Z \mid \exists F\ \text{coherent sheaf on } Z, \;\;
F_z\notin\perf(\sO_{Z,z}),\ \iota_*F\in\perf(X)\,\}.
\]
When \(2\) is invertible on \(X\) and \(W\) is a non-zero-divisor, he reconstructed the relative singular locus of the zero scheme \(X_0\) of \(W\) using the Balmer spectrum with a certain tensor structure.  
Our construction extends this result.

The author gratefully acknowledges Hirano and Ouchi, from whom he learned this idea.

\begin{thm}
Assume that \(X\) is separated and admits an ample family of line bundles.  
If \(W\) is a non-zero-divisor, then there is a homeomorphism
\[
\Spc^{\mo{fin}}(\mathrm{DMF}(X, L, W), \mathrm{DMF}(X, L, 0)) \cong \mathrm{Sing}(X_0/X).
\]
\end{thm}
\begin{proof}
This follows from Hirano's classification of \(\mathrm{DMF}(X, L, 0)\)-submodules~\cite{Hi}.
\end{proof}

Note that the underlying set of \(\Spc(\mathrm{DMF}(X, L, W), \mathrm{DMF}(X, L, 0))\) is the same, but its family of closed subsets is generated by \(\{\, Z \subset \mathrm{Sing}(X_0/X) \mid Z \text{ is closed in } X \,\}\).

\section*{Acknowledgements}

I am deeply grateful to Seidai Yasuda, Takumi Asano, and Kango Ito for their invaluable support during the preparation of this paper.  
I also thank Yuki Hirano, Daigo Ito, Hiroki Matsui, and Genki Ouchi for their helpful advice and insightful comments.

\bibliographystyle{amsalpha}
\bibliography{biblatex}

\providecommand{\bysame}{\leavevmode\hbox to3em{\hrulefill}\thinspace}
\providecommand{\MR}{\relax\ifhmode\unskip\space\fi MR }
\providecommand{\MRhref}[2]{%
  \href{http://www.ams.org/mathscinet-getitem?mr=#1}{#2}
}
\providecommand{\href}[2]{#2}
\begin{thebibliography}{BZFN10}

\bibitem[Bal05]{Ba}
Paul Balmer, \emph{The spectrum of prime ideals in tensor triangulated categories}, J. Reine Angew. Math. \textbf{588} (2005), 149--168.

\bibitem[BB03]{BB}
Alexei Bondal and Michel Van~den Bergh, \emph{Generators and representability of functors in commutative and noncommutative geometry}, Mosc. Math. J. \textbf{3} (2003), no.~1, 1--36.

\bibitem[Ber09]{Be}
Marcello Bernardara, \emph{A semiorthogonal decomposition for {Brauer-Severi} schemes}, Math. Nachr. \textbf{282} (2009), no.~10, 1406--1413.

\bibitem[BN93]{BN}
Marcel B{\"o}kstedt and Amnon Neeman, \emph{Homotopy limits in triangulated categories}, Compos. Math. \textbf{86} (1993), 209--234.

\bibitem[BZFN10]{BFN}
David Ben-Zvi, John Francis, and David Nadler, \emph{Integral transforms and {Drinfeld} centers in derived algebraic geometry}, J. Amer. Math. Soc. \textbf{23} (2010), no.~4, 909--966.

\bibitem[Cal00]{Ca}
Andrei~Horia Caldararu, \emph{Derived categories of twisted sheaves on {Calabi-Yau} manifolds}, Cornell University, 2000.

\bibitem[EP15]{EP}
Alexander~I. Efimov and Leonid Positselski, \emph{Coherent analogues of matrix factorizations and relative singularity categories}, Algebra Number Theory \textbf{9} (2015), no.~5, 1159--1292. \MR{3366002}

\bibitem[Hir19]{Hi}
Yuki Hirano, \emph{Relative singular locus and {B}almer spectrum of matrix factorizations}, Trans. Amer. Math. Soc. \textbf{371} (2019), no.~7, 4993--5021. \MR{3934475}

\bibitem[HKO24]{HKO}
Yuki Hirano, Martin Kalck, and Genki Ouchi, \emph{Length of triangulated categories}, arXiv:2404.07583 (2024).

\bibitem[HO22]{HO}
Yuki Hirano and Genki Ouchi, \emph{Prime thick subcategories on elliptic curves}, Pacific J. Math. \textbf{318} (2022), no.~1, 69--88.

\bibitem[HO24]{HO2}
\bysame, \emph{{Fourier-Mukai} loci of {K3} surfaces of picard number one}, arXiv:2405.01169 (2024).

\bibitem[IM24]{IM}
Daigo Ito and Hiroki Matsui, \emph{A new proof of the {Bondal-Orlov} reconstruction using matsui spectra}, arXiv:2405.16776 (2024).

\bibitem[Ito23]{It}
Daigo Ito, \emph{Gluing of {Fourier-Mukai} partners in a triangular spectrum and birational geometry}, arXiv:2309.08147 (2023).

\bibitem[Kra23]{Kr}
Henning Krause, \emph{Central support for triangulated categories}, Int. Math. Res. Not. IMRN (2023), no.~22, 19773--19800. \MR{4669814}

\bibitem[KS19]{KS}
Henning Krause and Greg Stevenson, \emph{The derived category of the projective line}, Spectral Structures and Topological Spectral Structures and Topological Methods in Mathematics (2019), 275--298.

\bibitem[LP25]{liu2025classifying}
Jian Liu and Josh Pollitz, \emph{Classifying thick subcategories over a koszul complex via the curved bgg correspondence}, arXiv preprint arXiv:2502.13806 (2025).

\bibitem[Lur09]{Lu1}
Jacob Lurie, \emph{Higher topos theory}, Annals of Mathematics Studies \textbf{170} (2009).

\bibitem[Lur17]{Lu2}
\bysame, \emph{Higher algebra}.

\bibitem[Mat21]{Ma1}
Hiroki Matsui, \emph{Prime thick subcategories and spectra of derived and singularity categories of noetherian schemes}, Pacific J. Math. \textbf{313} (2021), no.~2, 433--457.

\bibitem[Mat23]{Ma2}
\bysame, \emph{Triangular spectra and their applications to derived categories of noetherian schemes}, arXiv:2301.03168 (2023).

\bibitem[Mil92]{Miller1992finite}
Haynes Miller, \emph{Finite localizations}, Bol. Soc. Mat. Mexicana (2) \textbf{37} (1992), no.~1-2, 383--389.

\bibitem[MT18]{MT}
Hiroki Matsui and Ryo Takahashi, \emph{Construction of spectra of triangulated categories without tensor structure and applications to commutative rings}, arXiv:1811.06312 (2018).

\bibitem[Nee92]{Ne}
Amnon Neeman, \emph{The chromatic tower for \(d(r)\)}, Topology \textbf{31} (1992), 519--532.

\bibitem[Pol19]{Po}
Josh Pollitz, \emph{The derived category of a locally complete intersection}, Adv. Math. \textbf{354} (2019), 106752.

\bibitem[Pos11]{Pos}
Leonid Positselski, \emph{Two kinds of derived categories, {K}oszul duality, and comodule-contramodule correspondence}, Mem. Amer. Math. Soc. \textbf{212} (2011), no.~996, vi+133. \MR{2830562}

\bibitem[Ros96]{Ro}
Alexander~L Rosenberg, \emph{Reconstruction of schemes}, MPI Preprints Series \textbf{108} (1996).

\bibitem[Sch19]{Sc}
Peter Scholze, \emph{Lectures on condensed mathematics}.

\bibitem[Tho97]{Th}
Robert~W Thomason, \emph{The classification of triangulated subcategories}, Compos. Math. \textbf{105} (1997), no.~1, 1--27.

\bibitem[To{\"e}12]{To2}
Bertrand To{\"e}n, \emph{Derived {A}zumaya algebras and generators for twisted derived categories}, Invent. Math. \textbf{189} (2012), no.~3, 581--652. \MR{2957304}

\end{thebibliography}
\end{document}